\documentclass{amsart}
\usepackage{macros}
\standardsettings

\renewcommand{\radius}{\rho}

\draftfalse

\begin{document}
\title{Diophantine approximation in Banach spaces}

\authorlior\authordavid\authormariusz

\begin{abstract}
In this paper, we extend the theory of simultaneous Diophantine approximation to infinite dimensions. Moreover, we discuss Dirichlet-type theorems in a very general framework and define what it means for such a theorem to be optimal. We show that optimality is implied by but does not imply the existence of badly approximable points.
\end{abstract}

\maketitle

\section{Introduction}

\begin{definition}
\label{definitiondiophantinespace}
A \emph{Diophantine space} is a triple $(X,\QQ,H)$, where $X$ is a complete metric space, $\QQ$ is a dense subset of $X$, and $H:\QQ\to(0,+\infty)$.
\end{definition}

The prototypical example is the triple $(\R^d,\Q^d,H_\std)$, where $H_\std$ is the standard height function on $\Q^d$, i.e. $H_\std(\pp/q) = q$ assuming that $\gcd(p_1,\ldots,p_d,q) = 1$. Other (mostly implicit) examples may be found in \cite{CCM, DodsonEveritt, FKMS, FishmanSimmons1, FishmanSimmons3, FSU4, Kristensen} and the references therein.

This paper has two goals. The first is to clarify the theory of Dirichlet-type theorems on an abstract Diophantine space. Until now, it seems that there is no generally accepted definition of what it means for a Dirichlet-type theorem to be optimal; in each case where a Dirichlet-type theorem is proved, its optimality is demonstrated by producing points which are badly approximable with respect to the approximation function of the Dirichlet-type theorem. However, in Section \ref{sectionoptimality} we make a case for a wider notion of optimality, which is implied by but does not imply the existence of badly approximable points.

The second goal of this paper is to provide a complete theory of Diophantine approximation in the Diophantine space $(X,\Q\Lambda,H_\std)$, where $X$ is a Banach space, $\Lambda\leq X$ is a lattice, and $H_\std$ is the standard height function on $\Q\Lambda$ (precise definitions given below). This is related to the first goal since it turns out that when $\Lambda$ is a non-cobounded lattice, the optimal Dirichlet function of $(X,\Q\Lambda,H_\std)$ does not possess badly approximable points. Thus the theory of Diophantine approximation in Banach spaces gives a natural example of optimality failing to imply the existence of badly approximable points, justifying the clarification made in the first part.

{\bf Convention 1.} In the introduction, propositions which are proven later in the paper will be numbered according to the section they are proven in. Propositions numbered as 1.\# are either straightforward, proven in the introduction, or quoted from the literature.

{\bf Convention 2.} $x_n\tendsto n x$ means $x_n\to x$ as $n\to +\infty$.

{\bf Convention 3.} $\HD(S)$ is the Hausdorff dimension of a set $S$. $\HH^f(S)$ is the Hausdorff $f$-measure of a set $S$.

{\bf Acknowledgements.} The first-named author was supported in part by the Simons Foundation grant \#245708. The third-named author was supported in part by the NSF grant DMS-1001874. The authors thank Erez Nesharim for pointing out some typos which appeared in the published version and have been corrected here.

\newpage

\subsection{Dirichlet-type theorems on Diophantine spaces}
\label{subsectiondirichlettype}
\begin{definition}
A \emph{Dirichlet-type theorem} on a Diophantine space $(X,\QQ,H)$ is a true statement of the form
\[
\all\xx\in X \;\;\exists C_\xx > 0\;\; \exists (\rr_n)_1^\infty\text{ in $\QQ$} \text{ such that } \begin{cases} \rr_n\tendsto n \xx \text{ and}\\  \dist(\xx,\rr_n)\leq C_\xx \psi\circ H(\rr_n) \all n\in\N\end{cases},
\]
where $\psi:(0,+\infty)\to(0,+\infty)$. The function $\psi$ is called a \emph{Dirichlet function}. If the constant $C_\xx$ can be chosen to be independent of $\xx$, then the function $\psi$ is called \emph{uniformly Dirichlet}.\footnote{Most known Dirichlet functions are uniformly Dirichlet; however, a Dirichlet function which is not uniformly Dirichlet is given in \cite[Theorem 8.1]{FKMS}.}
\end{definition}

The prototypical example is Dirichlet's theorem, which states that for the Diophantine space $(\R^d,\Q^d,H_\std)$, the function $\psi(q) = q^{-(1 + 1/d)}$ is uniformly Dirichlet (the constant $C_\xx$ is $1$ for every $\xx\in\R^d$).

Dirichlet-type theorems are common in treatments of various Diophantine spaces; cf. the references given above. However, a Dirichlet-type theorem is usually not considered important unless it is \emph{optimal}, or unable to be improved by more than a constant factor. The optimality of a Dirichlet function is usually established by demonstrating the existence of \emph{badly approximable} points.

\begin{definition}
\label{definitionbadlyapproximable}
Let $(X,\QQ,H)$ be a Diophantine space, and let $\psi:(0,+\infty)\to(0,+\infty)$. A point $\xx\in X$ is said to be \emph{badly approximable} with respect to $\psi$ if
\begin{equation}
\label{badlyapproximable}
\exists \varepsilon > 0 \all \rr\in\QQ \;\; \dist(\rr,\xx) \geq \varepsilon\psi\circ H(\rr).
\end{equation}
The set of points in $X$ which are badly approximable with respect to $\psi$ will be denoted $\BA_\psi$, and its complement will be denoted $\WA_\psi$.
\end{definition}

The intuitive reason that the existence of badly approximable points implies optimality is that ``if there were a Dirichlet-type theorem which improved the Dirichlet-type theorem corresponding to $\psi$ by more than a constant, it would contradict the existence of badly approximable points''. We can make this intuition into a theorem, specifically the following theorem:

\begin{reptheorem}{theoremBAimpliesoptimality}
Let $(X,\QQ,H)$ be a Diophantine space. If $\psi:(0,+\infty)\to(0,+\infty)$ is any nonincreasing function and if $\BA_\psi\neq\emptyset$ and if $\phi:(0,+\infty)\to(0,+\infty)$ satisfies
\begin{repequation}{phipsi}
\frac{\phi}{\psi}\to 0,
\end{repequation}
then $\phi$ is not a Dirichlet function.
\end{reptheorem}

In this paper, we take the point of view that the conclusion of Theorem \ref{theoremBAimpliesoptimality} rather than its hypothesis is the true definition of optimality of a Dirichlet function in a Diophantine space. In other words, a Dirichlet function $\psi:(0,+\infty)\to(0,+\infty)$ is \emph{optimal} if there is no Dirichlet function $\phi:(0,+\infty)\to(0,+\infty)$ satisfying \eqref{phipsi}. The inequivalence of optimality and the existence of badly approximable points will be demonstrated in Theorem \ref{theoremdirichletnoncobounded} below. However, their equivalence in the case where $X$ is $\sigma$-compact will be demonstrated in Proposition \ref{propositionconverse}.

One could also conceivably define a Dirichlet-type theorem to be optimal if it implies all other Dirichlet-type theorems. This notion will be made rigorous in Section \ref{sectionoptimality}; however, it turns out to be too strong, and even in $(\R,\Q,H_\std)$ there are no Dirichlet functions which satisfy this strong notion of optimality. However, the notion can be refined by requiring that $\psi$ and $\phi$ lie in a Hardy field (see \sectionsymbol\ref{subsectionhardyfields}); in this case, the notion turns out to be equivalent to the notion of optimality defined above.

\subsection{The four main questions in Diophantine approximation}
\label{subsectiondiophantinequestions}
Given any Diophantine space $(X,\QQ,H)$, we will be interested in the following questions:

\begin{itemize}
\item[1.] (Dirichlet-type theorem) Find an optimal Dirichlet function for the Diophantine space. Is the set of badly approximable points for this Dirichlet function nonempty?
\item[2.] (Jarn\'ik--Schmidt type theorem) Given $\psi:(0,+\infty)\to(0,+\infty)$, what is the Hausdorff dimension of $\BA_\psi$?
\item[3.] (Jarn\'ik--Besicovitch type theorem) Given $\psi:(0,+\infty)\to(0,+\infty)$, what is the Hausdorff dimension of $\WA_\psi$?
\item[4.] (Khinchin-type theorem) Given $\psi:(0,+\infty)\to(0,+\infty)$, what are the measures of $\BA_\psi$ and $\WA_\psi$?
\end{itemize}

Note that the last question assumes the existence of a natural measure on the space $X$.

We will usually be satisfied if questions 2-4 can be answered for functions $\psi$ satisfying reasonable hypotheses, e.g. for $\psi$ in a Hardy field (see \sectionsymbol\ref{subsectionhardyfields}).

\begin{remark}
One can also ask whether $\BA_\psi$ or $\WA_\psi$ is generic in a topological sense, i.e. comeager. However, the question is trivial, as shown by the following proposition:
\end{remark}

\begin{proposition}
Let $(X,\QQ,H)$ be a Diophantine space. Then for any function $\psi:(0,+\infty)\to(0,+\infty)$, $\WA_\psi$ is comeager.
\end{proposition}
\begin{proof}
By writing
\[
\WA_\psi = \bigcap_{n = 1}^\infty \bigcup_{\rr\in\QQ} B\left(\rr,\frac{1}{n}\psi\circ H(\rr)\right),
\]
we see that $\WA_\psi$ is the intersection of countably many open dense sets.
\end{proof}

\begin{remark}
An example of a Diophantine space with no (reasonable) optimal Dirichlet function is given in \cite[Theorem 1.3]{FishmanSimmons3}. Even if an optimal Dirichlet function exists, we should not expect it to be unique without additional constraints; cf. Remark \ref{remarkuncountablymany}.
\end{remark}


\ignore{
\subsection{Diophantine isomorphisms}
\begin{definition}
A \emph{Diophantine isomorphism} between two Diophantine spaces $(X_1,\QQ_1,H_1)$ and $(X_2,\QQ_2,H_2)$ is a bi-Lipschitz map $f:X_1\to X_2$ satisfying $f(\QQ_1) = \QQ_2$ and
\[
H_1 \asymp_\times H_2\circ f.
\]
\end{definition}
We would like to take this opportunity to clarify the relationship between Diophantine isomorphisms and Diophantine theorems.

\begin{definition}
A function $\psi:(0,+\infty)\to(0,+\infty)$ is \emph{regular} if for every $C_1 > 0$ there exists $C_2 > 0$ such that for every $x,y\in(0,+\infty)$,
\[
|\log(x/y)| \leq C_1 \;\;\Rightarrow\;\; |\log(\psi(x)/\psi(y))| \leq C_2.
\]
\end{definition}

\begin{proposition}
Let $(X_1,\QQ_1,H_1)$ and $(X_2,\QQ_2,H_2)$ be Diophantine spaces and let $f:X_1\to X_2$ be a Diophantine isomorphism. If $\psi:(0,+\infty)\to(0,+\infty)$ is regular, then
\begin{itemize}
\item[(i)] $\psi$ is (optimal) (uniformly) Dirichlet for $(X_1,\QQ_1,H_1)$ if and only if $\psi$ is (optimal) (uniformly) Dirichlet for $(X_2,\QQ_2,H_2)$, and
\item[(ii)] $f(\BA_\psi(X_1,\QQ_1,H_1)) = \BA_\psi(X_2,\QQ_2,H_2)$.
\end{itemize}
\end{proposition}
}

\subsection{Diophantine approximation in Banach spaces}
\label{subsectionbanachdiophantine}
\begin{definition}
Let $X$ be a Banach space. A \emph{lattice} in $X$ is a subgroup $\Lambda\leq X$ such that
\begin{itemize}
\item[(I)] $\Lambda$ is (topologically) discrete, or equivalently,
\[
\varepsilon_\Lambda := \min_{\pp\in\Lambda\butnot\{\0\}}\|\pp\| > 0, \text{ and}
\]
\item[(II)] $\R\Lambda$ is dense in $X$, or equivalently, no proper closed subspace of $X$ contains $\Lambda$.
\end{itemize}
If $\Lambda\leq X$ is a lattice, the \emph{standard height function} $H_\std:\Q\Lambda\to \N$ is the function
\[
H_\std(\rr) = \min\{q\in\N: q\rr\in \Lambda\},
\]
i.e. $H_\std(\pp/q) = q$ if $\pp/q$ is in reduced form.
\end{definition}
\begin{remark}
Suppose that $X$ is separable. Then for a closed subgroup $\Lambda\leq X$, the following are equivalent (see \cite[Theorem 1.1]{ADG}):
\begin{itemize}
\item[(A)] $\Lambda$ is discrete,
\item[(B)] $\Lambda$ is locally compact and does not contain any one-dimensional subspace of $X$,
\item[(C)] $\Lambda$ is countable,
\item[(D)] $\Lambda$ is isomorphic to a (finite or infinite) direct sum of copies of $\Z$,
\item[(E)] $\Lambda$ is a free abelian group.
\end{itemize}
\end{remark}
Clearly, if $\Lambda\leq X$ is a lattice then $(X,\Q\Lambda,H_\std)$ is a Diophantine space. If $X = \R^d$ and $\Lambda = \Z^d$, then this Diophantine space is just the usual space $(\R^d,\Q^d,H_\std)$ studied in simultaneous Diophantine approximation. This example generalizes to infinite dimension in several different ways:

\begin{example}
\label{examplelp}
Fix $1\leq p < \infty$. Then $\Z^\infty := \{\pp\in\Z^\N: p_i = 0\text{ for all but finitely many $i\in\N$}\}$ is a lattice in $\ell^p(\N)$.
\end{example}

\begin{remark}
In Example \ref{examplelp}, we are \emph{not} approximating a point $\xx\in\ell^p(\N)$ by an arbitrary rational point $\rr\in\Q^\N\cap\ell^p(\N)$; rather, we are only approximating $\xx$ by those rational points with only finitely many nonzero coordinates. The reason for this is that there is no appropriate analogue of the ``LCM of the denominators'' for a rational point with infinitely many nonzero coordinates.
\end{remark}

Note that for $p = \infty$, $\Z^\infty$ is not a lattice in $\ell^\infty(\N)$, since it is contained in $c_0(\N)$, the set of all sequences in $\ell^\infty(\N)$ which tend to zero, which is a proper closed subspace of $\ell^\infty(\N)$. To get an example in $\ell^\infty(\N)$, we have two options: shrink the space or expand the lattice.

\begin{example}
\label{examplec0}
$\Z^\infty$ is a lattice in $c_0(\N)$.
\end{example}

\begin{example}
\label{examplelinfty}
$\Z^\N$ is a lattice in $\ell^\infty(\N)$.
\end{example}

We remark that although the space $\ell^\infty(\N)$ is not separable, this does not cause any additional complications in our arguments, which apply equally well to separable and non-separable Banach spaces.

It turns out that the theory of Diophantine approximation in $(X,\Q\Lambda,H_\std)$ depends on one crucial dichotomy: whether or not the lattice $\Lambda$ is cobounded. A lattice $\Lambda\leq X$ is \emph{cobounded} if its \emph{codiameter}
\[
\codiam(\Lambda) := \sup\{\dist(\xx,\Lambda):\xx\in X\}
\]
is finite. In the above, Examples \ref{examplec0} and \ref{examplelinfty} are cobounded, whereas Example \ref{examplelp} is not cobounded.



\subsubsection{Prevalence}
It is not clear what measure would be natural on an infinite-dimensional Banach space. In \cite{HSY} (see also \cite{Christensen}), B. R. Hunt, T. D. Sauer, and J. A. Yorke argued that asking for a measure is too much, and one should be satisfied with being able to give a good definition of ``full measure'' and ``measure zero''. They introduced the notions of \emph{shy} and \emph{prevalent} subsets of a Banach space:

\begin{definition}
Let $X$ be a Banach space. A measure $\mu$ is \emph{transverse} to a set $S\subset X$ if $\mu(S + \vv) = 0$ for all $\vv\in X$. $S$ is said to be \emph{shy} if it is transverse to some compactly supported probability measure, and \emph{prevalent} if its complement is shy.
\end{definition}

If $X$ is finite-dimensional, then a set is shy if and only if it has Lebesgue measure zero; it is prevalent if and only if its complement has Lebesgue measure zero. Moreover, the set of shy sets form a $\sigma$-ideal (i.e. the countable union of shy sets is shy, and any set contained in a shy set is shy). 
These facts together with several others (see \cite{HSY}) give support to the idea that ``shy'' is the appropriate analogue of ``measure zero'' in infinite dimensions and that ``prevalent'' is the appropriate analogue of ``full measure''.

In the sequel we will need the following proposition:

\begin{proposition}
\label{propositionHDinfinity}
Non-shy sets (and in particular prevalent sets) have full Hausdorff dimension.
\end{proposition}
\begin{proof}
Since the proposition is obvious if $\dim(X) < +\infty$, assume that $\dim(X) = +\infty$. Let $S\subset X$ be a non-shy set. Fix $n\in\N$, let $X_n\subset X$ be an $n$-dimensional subspace, and let $\mu_n$ be Lebesgue measure on the unit ball of $X_n$. Since $S$ is not shy, there exists $\vv\in X$ such that $\mu_n(S + \vv) > 0$. Since $\mu_n$ gives measure zero to any set of Hausdorff dimension strictly less than $n$, we have $\HD(S) = \HD(S + \vv) \geq n$. Since $n$ was arbitrary, $\HD(S) = +\infty$.
\end{proof}

\subsection{Main theorems}
\label{subsectionmaintheorems}

We now present the theory of Diophantine approximation in the space $(X,\Q\Lambda,H_\std)$, where $X$ is a Banach space and $\Lambda\leq X$ is a lattice. The theory breaks down into three major cases: finite-dimensional, infinite-dimensional cobounded, and infinite-dimensional non-cobounded. (In finite dimensions, every lattice is cobounded.)



\begin{notation}
For $s\geq 0$, let
\[
\psi_s(q) = q^{-s}.
\]
\end{notation}

\subsubsection{Finite-dimensional case}
Assume that $\Lambda$ is a lattice in a $d$-dimensional Banach space $X$, with $d < +\infty$. Then there exists a linear isomorphism $T:\R^d\to X$ such that $T[\Z^d] = \Lambda$. This demonstrates that the classical results quoted below hold for any lattice in any finite-dimensional Banach space, not just for $\Z^d\leq\R^d$.

\begin{theorem}[Dirichlet 1842 ($d\in\N$); optimality by Liouville 1844 ($d = 1$), Perron 1921 ($d\in\N$)]
\label{theoremdirichlet}
For every $\xx\in X$ and $Q\in\N$, there exists $\pp\in\Lambda$ and $q\leq Q$ such that
\[
\left\|\xx - \frac{\pp}{q}\right\| \leq \frac{C}{qQ^{1/d}},
\]
where $C > 0$ is independent of $\xx$. In particular, the function $\psi_{1 + 1/d}$ is uniformly Dirichlet, and in fact, $\psi_{1 + 1/d}$ is optimal.
\end{theorem}

\begin{theorem}[Jarn\'ik 1928 ($d = 1$), Schmidt 1969 ($d\in\N$)]
\label{theoremfulldimension}
We have $\HD(\BA_{\psi_{1 + 1/d}}) = d$.
\end{theorem}

\begin{theorem}[Jarn\'ik 1929 ($d = 1$), Jarn\'ik 1931 ($d\in\N$), Besicovitch 1934 ($d = 1$)]
\label{theoremjarnik}
For all $s\geq 1 + 1/d$, we have $\HD(\WA_{\psi_s}) = (d + 1)/s$.
\end{theorem}

\begin{theorem}[Khinchin 1924 ($d = 1$), Khinchin 1926 ($d\in\N$), Groshev 1938 ($d\in\N$)]
\label{theoremkhinchin}
If $q\mapsto q^d\psi(q)$ is nonincreasing, then $\WA_\psi$ is of full Lebesgue measure if the series $\sum_{q = 1}^\infty q^d\psi(q)$ diverges; if the series converges, then $\WA_\psi$ is of Lebesgue measure zero.
\end{theorem}


\subsubsection{Infinite-dimensional non-cobounded case}
Assume that $\Lambda$ is a non-cobounded lattice in an infinite-dimensional Banach space $X$.

\begin{reptheorem}{theoremdirichletnoncobounded}[Dirichlet-type theorem]
The function $\psi_0 \equiv 1$ is an optimal uniformly Dirichlet function. However, $\BA_{\psi_0} = \emptyset$.
\end{reptheorem}

\begin{reptheorem}{theoremkhinchinnoncobounded}[Khinchin-type theorem, Jarn\'ik--Schmidt type theorem]
For any function $\psi\to 0$, $\BA_\psi$ is prevalent. In particular, $\HD(\BA_\psi) = +\infty$.
\end{reptheorem}

To state the Jarn\'ik--Besicovitch type theorem in the non-cobounded case, we introduce the notion of \emph{strong discreteness}.

\begin{definition}
A lattice $\Lambda\leq X$ is \emph{strongly discrete} if
\[
\#(\Lambda\cap B(\0,C)) < +\infty \all C > 0.
\]
\end{definition}

All three of the examples given in \sectionsymbol\ref{subsectionbanachdiophantine} are not strongly discrete. \\


\begin{reptheorem}{theoremjarniknoncobounded}[Jarn\'ik--Besicovitch type theorem]
~
\begin{itemize}
\item[(i)] For any $s\geq 0$, we have $\HD(\WA_{\psi_s}) = +\infty$.
\item[(ii)] Suppose that $\Lambda$ is not strongly discrete. Then for any nonincreasing function $\psi\to 0$, $\HD(\WA_\psi) = +\infty$. In fact, for any nondecreasing function $f:(0,+\infty)\to(0,+\infty)$, $\HH^f(\WA_\psi) = +\infty$.
\end{itemize}
\end{reptheorem}


\subsubsection{Infinite-dimensional cobounded case}
Assume that $\Lambda$ is a cobounded lattice in an infinite-dimensional Banach space $X$.

\begin{reptheorem}{theoremdirichletcobounded}[Dirichlet-type theorem]
Fix $\varepsilon > 0$. For every $\xx\in X$ and for every $q\in\N$, there exists $\pp\in\Lambda$ such that
\[
\left\|\xx - \frac{\pp}{q}\right\| \leq \frac{\codiam(\Lambda) + \varepsilon}{q}\cdot
\]
In particular, the function $\psi_1(q) = 1/q$ is uniformly Dirichlet, and in fact, $\psi_1$ is optimal.
\end{reptheorem}

\begin{reptheorem}{theoremjarnikcobounded}[Jarn\'ik--Besicovitch type theorem]
For any nonincreasing function $\psi\to 0$, $\HD(\WA_\psi) = +\infty$. In fact, for any nondecreasing function $f:(0,+\infty)\to(0,+\infty)$, $\HH^f(\WA_\psi) = +\infty$.
\end{reptheorem}

\begin{reptheorem}{theoremkhinchincobounded}[Khinchin-type theorem, Jarn\'ik--Schmidt type theorem]
The set $\BA_{\psi_1}$ is prevalent. In particular, $\HD(\BA_{\psi_1}) = +\infty$.
\end{reptheorem}


\begin{remark}
Based on the finite-dimensional case, it is natural to expect that $\psi_1(q) = 1/q$ is an optimal Dirichlet function in the infinite-dimensional case, as it is the limit of the optimal Dirichlet functions $\psi_{1 + 1/d}$ of the finite-dimensional cases. However, according to the theorems above this is only true if the lattice is cobounded, whereas if the lattice is not cobounded then $\psi_0\equiv 1$ is the optimal Dirichlet function. A possible explanation for this can be found in the fact that in $\R^d$, the function $\psi_{1 + 1/d}$ is uniformly Dirichlet with the constant $C_d = 1$ if $\R^d$ is equipped with the $\ell^\infty$ norm; this suggests that if the $\ell^\infty$ norm is used, then there can be stability as $d\to\infty$. If an $\ell^p$ norm is used with $1\leq p < \infty$, then the constant $C_d$ will degenerate as $d\to\infty$, and the limit function will no longer be Dirichlet. (To look at it in another way, in order to ``take the limit of Dirichlet's theorem'' one would need to take the limit of the functions $C_d\psi_{1 + 1/d}$ as $d\to\infty$, and if $C_d\tendsto d \infty$ fast enough, then this sequence does not converge.)
\end{remark}

\draftnewpage
\section{Optimal Dirichlet functions}
\label{sectionoptimality}
In this section we discuss and motivate the notion of an optimal Dirichlet function introduced in \sectionsymbol\ref{subsectiondirichlettype}. We begin with the following observation:

\begin{observation}
\label{observationdirichletpartialorder}
Let $(X,\QQ,H)$ be a Diophantine space. Suppose that $\psi\leq C\phi$, with $\psi$ Dirichlet. Then $\phi$ is Dirichlet.
\end{observation}

Based on this observation, one is tempted to say that a Dirichlet function $\psi$ is optimal if it is maximal in the partial order on Dirichlet functions defined by
\[
\psi \succ \phi \;\;\Leftrightarrow\;\; \exists C > 0 \;\; \psi \leq C\phi,
\]
or equivalently, if every other Dirichlet function $\phi$ can be proved to be Dirichlet as a result of applying Observation \ref{observationdirichletpartialorder}.
\begin{definition}
A Dirichlet function $\psi$ is \emph{strongly optimal} if $\psi\succ\phi$ for every Dirichlet function $\phi$.
\end{definition}
This definition makes rigorous the idea that a Dirichlet-type theorem is optimal if it ``implies all other Dirichlet-type theorems (via Observation \ref{observationdirichletpartialorder})''. However, the definition is too strong even for the most canonical Diophantine space $(\R,\Q,H_\std)$. Indeed, we have the following:

\begin{proposition}
There is no strongly optimal Dirichlet function on $(\R,\Q,H_\std)$. In particular, the Dirichlet function $\psi_2$ is not strongly optimal on $(\R,\Q,H_\std)$.
\end{proposition}
\begin{proof}
\begin{lemma}
\label{lemmapsiQ}
For any sequence $\mathbf Q = (Q_n)_1^\infty$ increasing to infinity, the function
\begin{equation}
\label{psiqqQq}
\psi_{\mathbf Q}(q) = \frac{1}{q Q(q)},
\end{equation}
where
\[
Q(q) = \min\{Q_n: Q_n\geq q\}
\]
is uniformly Dirichlet for $(\R,\Q,H_\std)$.
\end{lemma}
\begin{subproof}
Fix $x\in\R$ and let $C = 1$. By Theorem \ref{theoremdirichlet}, for each $n\in\N$ there exists $r_n = p_n/q_n\in\Q$ with $q_n\leq Q_n$ such that
\begin{equation}
\label{xrqQn}
|x - r_n| \leq \frac{1}{q_nQ_n}\cdot
\end{equation}
Since $Q_n\geq q_n$, we have
\[
Q(q_n)\leq Q_n
\]
and thus
\begin{equation}
\label{xrqQq}
|x - r_n| \leq \frac{1}{q_nQ(q_n)} = \psi(q_n).
\end{equation}
Since $Q_n\tendsto n +\infty$, \eqref{xrqQn} implies that $r_n\tendsto n x$. Thus the function \eqref{psiqqQq} is uniformly Dirichlet.

\end{subproof}
To complete the proof, we will find two sequences $\mathbf Q_0 = (Q_n^{(0)})_1^\infty$ and $\mathbf Q_1 = (Q_n^{(1)})_1^\infty$ such that the minimum of the two functions $\psi_{\mathbf Q_0}$ and $\psi_{\mathbf Q_1}$ is not a Dirichlet function. We choose the sequences
\[
Q_n^{(i)} = 2^{2^{2n + i}}, \;\; i = 1,2
\]
and leave it to the reader to verify that the function $\phi = \min(\psi_{\mathbf Q_0},\psi_{\mathbf Q_1})$ satisfies
\[
\phi \leq \psi_3.
\]
(It suffices to check the inequality for the worst-case scenario $q \in \mathbf Q_0\cup \mathbf Q_1$.) Now suppose that $\psi$ is an optimal Dirichlet function for $(\R,\Q,H_\std)$. Then since $\psi_{\mathbf Q_0}$ and $\psi_{\mathbf Q_1}$ are Dirichlet, we have $\psi \leq C \psi_{\mathbf Q_i}$ for some $C > 0$. Thus $\psi\leq C\phi \leq C\psi_3$, so by Observation \ref{observationdirichletpartialorder}, $\psi_3$ is Dirichlet. This contradicts the optimality of $\psi_2$, since $\frac{\psi_3}{\psi_2}\to 0$.
\end{proof}

Having ruled out strong optimality as a notion of optimality, we turn to the weaker notion of optimality given in \sectionsymbol\ref{subsectiondirichlettype}. We repeat it here for convenience:

\begin{definition}
\label{definitiondirichletoptimal}
A Dirichlet function $\psi$ is \emph{optimal} (with respect to a Diophantine space $(X,\QQ,H)$) if there is no Dirichlet function $\phi$ satisfying
\begin{equation}
\label{phipsi}
\frac{\phi}{\psi}\to 0.
\end{equation}
\end{definition}

How do we know that this is the ``correct'' definition? We give two reasons:

\begin{itemize}
\item[1.] In the case of a $\sigma$-compact Diophantine space, for example a finite-dimensional Banach space, our new definition agrees with the more classical criterion of the existence of badly approximable points. Even in the non $\sigma$-compact case, the existence of badly approximable points implies optimality.
\item[2.] The notion of optimality agrees with the notion of strong optimality if the class of functions is restricted to a suitable class of ``non-pathological'' functions.
\end{itemize}

We now proceed to elaborate on each of these reasons.

\subsection{Optimality versus BA}
Traditionally, the existence of badly approximable points has been thought to demonstrate that Dirichlet's function is optimal (up to a constant). In our terminology, this intuition becomes a theorem:


\begin{theorem}[Existence of BA implies optimality]
\label{theoremBAimpliesoptimality}
Let $(X,\QQ,H)$ be a Diophantine space. If $\psi:(0,+\infty)\to(0,+\infty)$ is any nonincreasing function and if $\BA_\psi\neq\emptyset$ and if $\phi:(0,+\infty)\to(0,+\infty)$ satisfies \eqref{phipsi}, then $\phi$ is not a Dirichlet function.
\end{theorem}
\begin{proof}
Fix $\xx\in\BA_\psi$. If $\phi$ is Dirichlet, then there exist $C_\xx > 0$ and a sequence $(\rr_n)_1^\infty$ such that
\[
\dist(\rr_n,\xx) \leq C_\xx\phi(q_n) \text{ and } \rr_n\tendsto n \xx,
\]
where $q_n := H(\rr_n)$. Combining with \eqref{badlyapproximable} gives
\[
\varepsilon\psi(q_n)\leq C_\xx\phi(q_n);
\]
rearranging yields
\begin{equation}
\label{boundedfrombelow}
\frac{\phi(q_n)}{\psi(q_n)} \geq \frac{\varepsilon}{C_\xx} > 0.
\end{equation}
On the other hand, we have
\[
\varepsilon\psi(q_n)\leq \dist(\rr_n,\xx)\tendsto n 0;
\]
since $\psi$ is positive and nonincreasing this implies that $q_n\tendsto n +\infty$. Together with \eqref{boundedfrombelow}, this contradicts \eqref{phipsi}.
\end{proof}

The converse to Theorem \ref{theoremBAimpliesoptimality} does not hold in such generality (cf. Theorem \ref{theoremdirichletnoncobounded}), but rather holds only under the hypothesis that the underlying Diophantine space is $\sigma$-compact.


\begin{proposition}[Optimality implies existence of BA]
\label{propositionconverse}
Let $(X,\QQ,H)$ be a $\sigma$-compact Diophantine space. Then if $\psi$ is a bounded optimal Dirichlet function, then $\BA_\psi\neq \emptyset$.
\end{proposition}
\begin{proof}
Let $(K_n)_1^\infty$ be an increasing sequence of compact sets whose union is $X$.

Suppose by contradiction that $\BA_\psi = \emptyset$. Then for each $n\in\N$ and for each $\xx\in X$, there exists $\rr\in\QQ$ such that
\[
\dist(\rr,\xx) < \frac{1}{n}\psi\circ H(\rr).
\]
Let $U_{\rr,n}$ be the set of all $\xx$ satisfying the above; then for each $n\in\N$, $(U_{\rr,n})_{\rr}$ is an open cover of $X$, and in particular an open cover of $K_n$. Let $(U_{\rr,n})_{\rr\in F_n}$ be a finite subcover, and let $Q_n = \max_{F_n}(H)$. Let
\[
\phi(q) = \psi(q)\max\{1/n:q \leq Q_n\}.
\]
Clearly $\phi(q)/\psi(q)\tendsto q 0$. We claim that $\phi$ is a Dirichlet function. Indeed, fix $\xx\in X$, and let $C_\xx = 1$. For all $n\in\N$ sufficiently large, we have $\xx\in K_n$. Fix such an $n$, and choose $\rr_n\in F_n$ so that $\xx\in U_{\rr_n,n}$. Then $q_n := H(\rr_n)\leq Q_n$. It follows that
\[
\phi(q_n)\geq \frac{1}{n}\psi(q_n) > \dist(\rr_n,\xx).
\]
Since $\psi$ is bounded, this implies that $\rr_n\tendsto n \xx$. Thus $\xx$ is $\phi$-approximable. Thus $\phi$ is a Dirichlet function, and so $\psi$ is not an optimal Dirichlet function.

\ignore{
Suppose by contradiction that $\BA_\psi = \emptyset$. Then for each $\varepsilon > 0$ and for each $x\in X$, there exists $\rr\in\QQ$ such that
\[
\dist(x,\rr) < \varepsilon\psi(q).
\]
Let $U_{\rr,\varepsilon}$ be the set of all $x$ satisfying the above; then for each $\varepsilon > 0$, $(U_{\rr,\varepsilon})_{\rr}$ is an open cover of $X$. Let $(U_{\rr,\varepsilon})_{\rr\in F_\varepsilon}$ be a finite subcover, and let $Q_\varepsilon = \max\{q:\rr\in F_\varepsilon\}$. Let
\[
\phi(q) = \psi(q)\max\{\varepsilon > 0:q \leq Q_\varepsilon\}.
\]
Clearly $\phi(q)/\psi(q)\tendsto q 0$. We claim that every point of $X$ is $\phi$-approximable, which demonstrates that $\psi$ is not an optimal Dirichlet function. Indeed, fix $x\in X$. Fix $\varepsilon > 0$, and choose $\rr\in F_\varepsilon$ so that $x\in U_{\rr,\varepsilon}$. Then $q\leq Q_\varepsilon$. It follows that
\[
\phi(q)\geq \varepsilon\psi(q) > \dist(x,\rr).
\]
As $\varepsilon$ tends to zero, the $\rr$ such that the above holds must be distinct. This implies that $x$ is $\phi$-approximable.
}
\end{proof}

\subsection{Hardy fields}
\label{subsectionhardyfields}
One possible reaction to the phenomenon of Lemma \ref{lemmapsiQ} is to insist that the functions $\psi_{\mathbf Q}$ defined in that lemma are pathological. One way to make this rigorous is to consider the notion of a \emph{Hardy field}.

\begin{definition}
A \emph{germ at infinity} is an equivalence class of $C^\infty$ functions from $\Rplus$ to $\R$, where two functions are considered equivalent if they agree on all sufficiently large values.

A \emph{Hardy field} is a field of germs at infinity which is closed under differentiation.
\end{definition}

\begin{remark}
If $\psi\not\equiv 0$ is an element of a Hardy field, then by definition, there is a $C^\infty$ function from $\Rplus$ to $\R$ which agrees with $1/\psi$ on all sufficiently large values. This implies that $\psi\neq 0$ on all sufficiently large values; since $\psi$ is continuous, either $\psi > 0$ or $\psi < 0$ on all sufficiently large values.
\end{remark}

A standard example of a Hardy field is the class of \emph{Hardy $L$-functions}, which is the class all functions which can be written using the symbols $+,-,\times,\div,\exp$ and $\log$ together with the constants and the identity function; cf. \cite[Chapter III]{Hardy}. From now on, we will consider functions to be ``non-pathological'' if their germs at infinity are elements of some fixed Hardy field.


\begin{observation}
\label{observationhardyfield}
If the germs of $\psi$ and $\phi$ are elements of the same Hardy field, then either $\psi\prec\phi$ or $\phi\prec\psi$. Moreover, $\psi\not\succ\phi$ if and only if $\frac{\phi}{\psi}\to 0$.
\end{observation}
\begin{proof}
Both assertions follow from the well-known fact that $\lim_{q\to\infty}\frac{\phi}{\psi}(q)$ exists.
\end{proof}

%

The second part of Observation \ref{observationhardyfield} can be taken as a motivation for Definition \ref{definitiondirichletoptimal}. Indeed, it shows that if a Hardy field is fixed and all functions are assumed to be elements of that Hardy field, then the notions of strong optimality and optimality agree.

\begin{remark}
\label{remarkuncountablymany}
The first part of Observation \ref{observationhardyfield} shows that if a Hardy field is fixed and all functions are assumed to be elements of that Hardy field, then any two optimal Dirichlet functions $\phi$ and $\psi$ ``agree up to a constant'', i.e. their ratio $\frac{\phi}{\psi}$ is bounded from above and below. If the restriction to a Hardy field is not made, then Lemma \ref{lemmapsiQ} can be used to show that there are uncountably many optimal Dirichlet functions on $(\R,\Q,H_\std)$, no two of which are comparable. (The Dirichlet function $\psi_{\mathbf Q}$ is optimal because $\psi_{\mathbf Q}\leq \psi_2$.)
\end{remark}

\begin{remark}
Restricting to elements of a Hardy field is also useful in answering questions 2-4 of \sectionsymbol\ref{subsectiondiophantinequestions}. To see this, note that the map $\psi\mapsto\BA_\psi$ is order-preserving, i.e. $\psi\prec\phi$ implies $\BA_\psi\subset\BA_\phi$. Similarly, the map $\psi\mapsto\WA_\psi$ is order-reversing. Since in a Hardy field, $\prec$ is a total order (Observation \ref{observationhardyfield}), it is possible to prescribe the values of $\HD(\BA_\psi)$ and $\HD(\WA_\psi)$ on all $\psi$ in a Hardy field by prescribing the values of $\HD(\BA_\psi)$ and $\HD(\WA_\psi)$ for a relatively small collection of $\psi$s. Using this principle, in the case of Banach spaces it is possible to answer questions 2-4 completely (except for question 3 in the case of strongly discrete lattices) based on the information given in \sectionsymbol\ref{subsectionmaintheorems}. Details are left to the reader.
\end{remark}

\ignore{
\section{New Section}

\begin{theorem}
Fix $\varepsilon > 0$. For every $\xx\in X$ and for every $q\in\N$, there exists $\pp\in\Lambda$ such that
\[
\left\|\xx - \frac{\pp}{q}\right\| \leq \frac{\codiam(\Lambda) + \varepsilon}{q}\cdot
\]
\end{theorem}
The proof is trivial; simply note that $\dist(q\xx,\Lambda)$ is a member of the supremum defining $\codiam(\Lambda)$, then divide by $q$. The $\varepsilon$ term is necessary because we cannot know whether or not the infimum defining $\dist(q\xx,\Lambda)$ is achieved.

By itself, Theorem \ref{theoremdirichletcobounded} is not interesting; in fact, it is true in finite dimensions as well. Nevertheless, we have the following:

\begin{theorem}
\label{theoremkhinchincobounded}

\end{theorem}

we can test the optimality of Theorem \ref{theoremdirichletcobounded} by letting
\[
\BA_\Lambda := \left\{\xx\in X: \exists \varepsilon > 0 \all \pp\in\Lambda \all q\in\N \;\; \left\|\xx - \frac{\pp}{q}\right\| \geq \frac{\varepsilon}{q}\right\}
\]
and stating that if $\BA_\Lambda \neq \emptyset$, then Theorem \ref{theoremdirichletcobounded} cannot be improved by more than a constant. (A more rigorous discussion is given below in \ref{}.)

\begin{itemize}
\item $X$ is a Banach space, and
\item $\Lambda\leq X$ is a \emph{lattice}, i.e. discrete subgroup of $X$ satisfying $\cl{\R\Lambda} = X$, and

\item $H_\std:\Q\Lambda\to\N$ is the standard height function (see Example \ref{examplebanachdiophantine}).
\end{itemize}

\begin{proposition}
\label{propositionHDinfinity}
Non-shy sets (and in particular prevalent sets) have Hausdorff dimension $+\infty$.
\end{proposition}
\begin{proof}
Let $S\subset X$ be a non-shy set. Fix $n\in\N$, and let $\mu_n$ be an Ahlfors $n$-regular measure on $X$ with compact support (for example, Lebesgue measure on an appropriate subset). Since $S$ is not shy, there exists $\vv\in X$ such that $\mu_n(S + \vv) > 0$. Then $\HD(S) = \HD(S + \vv) \geq n$. Since $n$ was arbitrary, $\HD(S) = +\infty$.
\end{proof}

\section{Diophantine spaces}

\begin{definition}
\label{definitiondiophantinespace}
A \emph{Diophantine space} is triple $(X,\QQ,H)$, where $X$ is a complete metric space, $\QQ$ is a dense subset of $X$, and $H:\QQ\to(0,+\infty)$.
\end{definition}

The canonical example is the following:
\begin{example}
\label{exampleRd}
Fix $d\in\N$, and let $H_\std:\Q^d\to\N$ be the standard height function, i.e. for each $\rr\in\Q^d$, $q = H_\std(\rr)$ is the smallest integer such that $q\rr\in\Z^d$. Then $(\R^d,\Q^d,H_\std)$ is a Diophantine space.
\end{example}

The papers \cite{CCM, DodsonEveritt, DodsonKemble, FKMS, FishmanSimmons1, FishmanSimmons3, FSU4, Kristensen} can all be viewed as talking implicitly about Diophantine spaces other than $(\R^d,\Q^d,H_\std)$. In this paper we will be interested in the following generalization of Example \ref{exampleRd} to the setting of Banach spaces:

\begin{example}
\label{examplebanachdiophantine}
Let $X$ be a Banach space. We will call a set $\Lambda\leq X$ be a \emph{lattice} if
\begin{itemize}
\item[(I)] $\Lambda$ is an additive subgroup of $X$,
\item[(II)] $\Lambda$ is (topologically) discrete, and
\item[(III)] $\R\Lambda$ is dense in $X$, or equivalently, no proper closed subspace of $X$ contains $\Lambda$.
\end{itemize}
Given any lattice $\Lambda\leq X$, we define the \emph{standard height function} $H_\std:\Q\Lambda\to\N$ in the same way as above, i.e. for each $\rr\in\Q\Lambda$, $q = H_\std(\rr)$ is the smallest integer such that $q\rr\in\Lambda$. Then $(X,\Q\Lambda,H_\std)$ is a Diophantine space.
\end{example}
Example \ref{exampleRd} is the special case of Example \ref{examplebanachdiophantine} which occurs when $X = \R^d$ and $\Lambda = \Z^d$. For more special cases, see Section \ref{sectionexamples} below.

The classical theorems of Diophantine approximation can all be reformulated as statements about the Diophantine spaces $(\R^d,\Q^d,H_\std)$. In particular, to reformulate Dirichlet's theorem we make the following definition:

\begin{definition}
\label{definitiondirichletfunction}
Let $(X,\QQ,H)$ be a Diophantine space. A function $\psi:\Rplus\to(0,+\infty)$ is called a \emph{Dirichlet function} for $(X,\QQ,H)$ if for every $\xx\in X$, there exists $C = C_\xx > 0$ and a sequence $(\rr_n)_1^\infty$ such that
\[
\dist(\rr_n,\xx) \leq C \psi\circ H(\rr_n) \text{ and } \rr_n\tendsto n \xx.
\]
\end{definition}
With this terminology, Dirichlet's classical theorem is equivalent (up to a constant) to the following:
\begin{theorem}[Dirichlet's theorem, reformulated]
\label{theoremdirichletreformulated}
Fix $d\in\N$. Then the function
\begin{equation}
\label{dirichletstandard}
\psi(q) = q^{-(1 + 1/d)}
\end{equation}
is a Dirichlet function for the Diophantine space $(\R^d,\Q^d,H_\std)$.
\end{theorem}
\begin{remark}
\label{remarkC}
In Theorem \ref{theoremdirichletreformulated}, we may let $C = 1$ (or $C = 1/\sqrt 5$ in the case $d = 1$) in Definition \ref{definitiondirichletfunction}. However, there are natural examples where $C_\xx$ depends on $\xx$; see \cite{FKMS, FSU4}.
\end{remark}
\begin{remark}
\label{remarktrivial}
Trivially, any Diophantine space has $\psi\equiv 1$ as a Dirichlet function.
\end{remark}

To generalize the notion of badly approximable points we make the following definition:
\begin{definition}
\label{definitionbadlyapproximable}
Let $\psi:\Rplus\to(0,+\infty)$ be any function. A point $\xx\in X$ is called \emph{badly approximable} with respect to $\psi$ if there exists $\varepsilon > 0$ such that for all $\rr\in\QQ$ we have
\begin{equation}
\label{badlyapproximable}
\dist(\rr,\xx) \geq \varepsilon\psi\circ H(\rr).
\end{equation}
The set of points which are badly approximable with respect to $\psi$ will be denoted $\BA_\psi$, and its complement will be denoted $\WA_\psi$.
\end{definition}
Note that $\BA_\psi$ has been defined for any function $\psi$, although it makes the most sense when $\psi$ is a Dirichlet function.

The classical theorem of Jarn\'ik can now be restated as follows:
\begin{theorem}
\label{theoremjarnikreformulated}
Fix $d\in\N$. Then for the Diophantine space $(\R^d,\Q^d,H_\std)$ we have
\[
\HD(\BA_\psi) = d,
\]
where $\psi$ is as in \eqref{dirichletstandard}.
\end{theorem}

\subsection{Optimality vs. the existence of badly approximable points}
Theorem \ref{theoremjarnikreformulated} is usually taken as demonstrating the optimality of the Dirichlet function \eqref{dirichletstandard}, since it demonstrates that $\BA_\psi\neq\emptyset$, and therefore that ``Dirichlet's theorem cannot be improved by more than a constant''. But what precisely is meant by this? Here is one possible definition:
\begin{definition}
\label{definitiondirichletoptimal}
A Dirichlet function $\psi$ is \emph{optimal} (with respect to a Diophantine space $(X,\QQ,H)$) if there is no Dirichlet function $\phi$ satisfying
\begin{equation}
\label{phipsi}
\frac{\phi(q)}{\psi(q)}\tendsto q 0.
\end{equation}
\end{definition}
According to this definition, the existence of badly approximable points does indeed prove the optimality of \eqref{dirichletstandard}:

\begin{proposition}[Existence of BA implies optimality]
\label{propositionBAimpliesoptimality}
Let $(X,\QQ,H)$ be a Diophantine space. If $\psi:\Rplus\to(0,+\infty)$ is any nonincreasing function and if $\BA_\psi\neq\emptyset$, then \eqref{phipsi} does not hold for any Dirichlet function $\phi$. In particular, if $\psi$ is a Dirichlet function for which $\BA_\psi\neq\emptyset$, then $\psi$ is optimal.
\end{proposition}
\begin{proof}
By contradiction, let $\phi$ be a Dirichlet function satisfying \eqref{phipsi}. Fix $\xx\in\BA_\psi$. Since $\phi$ is a Dirichlet function, there exist $C_\xx > 0$ and a sequence $(\rr_n)_1^\infty$ such that
\[
\dist(\rr_n,\xx) \leq C_\xx\phi(q_n) \text{ and } \rr_n\tendsto n \xx,
\]
where $q_n := H(\rr_n)$. Combining with \eqref{badlyapproximable} gives
\[
\varepsilon\psi(q_n)\leq C_\xx\phi(q_n);
\]
rearranging yields
\begin{equation}
\label{boundedfrombelow}
\frac{\phi(q_n)}{\psi(q_n)} \geq \frac{\varepsilon}{C_\xx} > 0.
\end{equation}
On the other hand, we have
\[
\varepsilon\psi(q_n)\leq \dist(\rr_n,\xx)\tendsto n 0;
\]
since $\psi$ is positive and nonincreasing this implies that $q_n\tendsto n +\infty$. Together with \eqref{boundedfrombelow}, this contradicts that $\frac{\phi}{\psi}\to 0$.
\end{proof}
The converse to Theorem \ref{theoremBAimpliesoptimality}, namely that if a Dirichlet function is optimal, then there exists at least one badly approximable point, is only true under the additional hypothesis that the space $X$ is $\sigma$-compact. Indeed, we shall show in Corollary \ref{corollarynoBA} below that for each $1\leq p < +\infty$, the function $\psi \equiv 1$ is an optimal Dirichlet function for the Diophantine space $(\ell^p(\R),\Q^\infty,H_\std)$ with no badly approximable points. (We do not claim that this is the easiest example of such a space, but it seems to be the first ``naturally occuring'' one.)

\begin{theorem}
\label{theoremconverse}
Let $(X,\QQ,H)$ be a $\sigma$-compact Diophantine space. Then if $\psi$ is a bounded optimal Dirichlet function, then $\BA_\psi\neq \emptyset$.
\end{theorem}
\begin{proof}
Let $(K_n)_1^\infty$ be an increasing sequence of compact sets whose union is $X$.

Suppose by contradiction that $\BA_\psi = \emptyset$. Then for each $n\in\N$ and for each $\xx\in X$, there exists $\rr\in\QQ$ such that
\[
\dist(\rr,\xx) < \frac{1}{n}\psi\circ H(\rr).
\]
Let $U_{\rr,n}$ be the set of all $\xx$ satisfying the above; then for each $n\in\N$, $(U_{\rr,n})_{\rr}$ is an open cover of $X$, and in particular an open cover of $K_n$. Let $(U_{\rr,n})_{\rr\in F_n}$ be a finite subcover, and let $Q_n = \max_{F_n}(H)$. Let
\[
\phi(q) = \psi(q)\max\{1/n:q \leq Q_n\}.
\]
Clearly $\phi(q)/\psi(q)\tendsto q 0$. We claim that $\phi$ is a Dirichlet function. Indeed, fix $\xx\in X$, and let $C_\xx = 1$. For all $n\in\N$ sufficiently large, we have $\xx\in K_n$. Fix such an $n$, and choose $\rr_n\in F_n$ so that $\xx\in U_{\rr_n,n}$. Then $q_n := H(\rr_n)\leq Q_n$. It follows that
\[
\phi(q_n)\geq \frac{1}{n}\psi(q_n) > \dist(\rr_n,\xx).
\]
Since $\psi$ is bounded, this implies that $\rr_n\tendsto n \xx$. Thus $\xx$ is $\phi$-approximable. Thus $\phi$ is a Dirichlet function, and so $\psi$ is not an optimal Dirichlet function.

\ignore{
Suppose by contradiction that $\BA_\psi = \emptyset$. Then for each $\varepsilon > 0$ and for each $x\in X$, there exists $\rr\in\QQ$ such that
\[
\dist(x,\rr) < \varepsilon\psi(q).
\]
Let $U_{\rr,\varepsilon}$ be the set of all $x$ satisfying the above; then for each $\varepsilon > 0$, $(U_{\rr,\varepsilon})_{\rr}$ is an open cover of $X$. Let $(U_{\rr,\varepsilon})_{\rr\in F_\varepsilon}$ be a finite subcover, and let $Q_\varepsilon = \max\{q:\rr\in F_\varepsilon\}$. Let
\[
\phi(q) = \psi(q)\max\{\varepsilon > 0:q \leq Q_\varepsilon\}.
\]
Clearly $\phi(q)/\psi(q)\tendsto q 0$. We claim that every point of $X$ is $\phi$-approximable, which demonstrates that $\psi$ is not an optimal Dirichlet function. Indeed, fix $x\in X$. Fix $\varepsilon > 0$, and choose $\rr\in F_\varepsilon$ so that $x\in U_{\rr,\varepsilon}$. Then $q\leq Q_\varepsilon$. It follows that
\[
\phi(q)\geq \varepsilon\psi(q) > \dist(x,\rr).
\]
As $\varepsilon$ tends to zero, the $\rr$ such that the above holds must be distinct. This implies that $x$ is $\phi$-approximable.
}
\end{proof}

\ignore{
\subsection{Non-uniqueness of the optimal Dirichlet function}
In this subsection we demonstrate that the optimal Dirichlet function \eqref{dirichletstandard} is not unique in any reasonable sense. Obviously, any function asymptotic to an optimal Dirichlet function is also an optimal Dirichlet function. However, we give below a class of examples of optimal Dirichlet functions for $(\R,\Q,H_\std)$ which are not asymptotic to each other. In fact, their minimum is not a Dirichlet function, which implies that they are not asymptotic, because if they were then their minimum would also be asymptotic to both of them, and would therefore also be an optimal Dirichlet function.\comdavid{I wasn't sure whether it made sense to include this result in this paper, but I am writing it here so that you can see it.}
\begin{proposition}
For the Diophantine space $(\R,\Q,H_\std)$, there exist two optimal Dirichlet functions whose minimum is not a Dirichlet function.
\end{proposition}
\begin{proof}
\begin{lemma}
For any sequence $\mathbf Q = (Q_n)_1^\infty$ tending to infinity, the function
\begin{equation}
\label{psiqqQq}
\psi_{\mathbf Q}(q) = \frac{1}{q Q(q)},
\end{equation}
where
\[
Q(q) = \min\{Q_n: Q_n\geq q\}
\]
is an optimal Dirichlet function.
\end{lemma}
\begin{subproof}
Fix $x\in\R$ and let $C = 1$. By the Theorem \ref{theoremdirichlet}, for each $n\in\N$ there exists $r_n = p_n/q_n\in\Q$ with $q_n\leq Q_n$ such that
\begin{equation}
\label{xrqQn}
\dist(x,r_n) \leq \frac{1}{q_nQ_n}\cdot
\end{equation}
Since $Q_n\geq q_n$, we have
\[
Q(q_n)\leq Q_n
\]
and thus
\begin{equation}
\label{xrqQq}
\dist(x,r_n) \leq \frac{1}{q_nQ(q_n)} = \psi(q_n).
\end{equation}
Since $Q_n\tendsto n +\infty$, \eqref{xrqQn} implies that $r_n\tendsto n x$. Thus the function \eqref{psiqqQq} is a Dirichlet function.

To see that the Dirichlet function \eqref{psiqqQq} is optimal, simply note that it is bounded above by the Dirichlet function \eqref{dirichletstandard}, which is optimal.
\end{subproof}
To complete the proof, we must find two sequences $\mathbf Q_0 = (Q_n^{(0)})_1^\infty$ and $\mathbf Q_1 = (Q_n^{(1)})_1^\infty$ such that the minimum of the two functions $\psi_{\mathbf Q_0}$ and $\psi_{\mathbf Q_1}$ is not a Dirichlet function. We choose the sequences
\[
Q_n^{(i)} = 2^{2^{2n + i}},
\]
and leave it to the reader to verify that the minimum $\psi$ of the functions constructed in this way satisfies
\[
\psi(q) \leq \frac{1}{q^3}
\]
and is therefore not a Dirichlet function for $(\R,\Q,H_\std)$.
\end{proof}
}

\subsection{The Jarn\'ik--Besicovitch and Khinchin theorems}

\begin{definition}
Let $(X,\QQ,H)$ be a Diophantine space. For any point $\xx\in X$, we can define the \emph{exponent of irrationality}
\[
\omega(\xx) := \limsup_{\substack{\rr\in\QQ \\ \rr\to\xx}}\frac{-\log\dist(\rr,\xx)}{\log H(\rr)}.
\]
For each $c > 0$ let
\[
W_c := \{\xx\in X:\omega(\xx)\geq c\}.
\]
We define the exponent of irrationality of the space $X$ to be 
\[
\omega(X) := \sup\{c\geq 0: W_c = X\} = \inf\{\omega(\xx):\xx\in X\}.
\]
Let
\begin{align*}
\VWA &:= \{\xx\in X:\omega(\xx) > \omega(X)\} = \bigcup_{c > \omega(X)}W_c\\
\Liou &:= \{\xx\in X:\omega(\xx) = +\infty\} = \bigcap_{c > 0}W_c
\end{align*}
be the set of \emph{very well approximable} numbers and \emph{Liouville} numbers, respectively.
\end{definition}

\begin{theorem}[The Jarn\'ik--Besicovitch theorem, reformulated]
\label{theoremjarnikbesicovitchreformulated}
Fix $d\in\N$, and consider the Diophantine space $(\R^d,\Q^d,H_\std)$. For every $c > \omega(\R^d) = 1 + 1/d$, we have
\[
\HD(W_c) = \frac{d + 1}{c}.
\]
In particular $\HD(\VWA) = d$ and $\HD(\Liou) = 0$.
\end{theorem}

\begin{proposition}
Let $(X,\QQ,H)$ be a Diophantine space. Then for any function $\psi:\Rplus\to(0,+\infty)$, $\WA_\psi$ is comeager.
\end{proposition}
\begin{proof}
By writing
\[
\WA_\psi = \bigcap_{n = 1}^\infty \bigcup_{\rr\in\QQ} B\left(\rr,\frac{1}{n}\psi\circ H(\rr)\right),
\]
we see that $\WA_\psi$ is the intersection of countably many open dense sets.
\end{proof}

}

\ignore{
\section{Diophantine approximation in $\ell^p(\R)$}

\subsection{Proof of Theorems \ref{theoremLpkhinchin} and \ref{theoremnoBA}}

\begin{theorem}
\label{theoremLpkhinchin}
Fix $1\leq p < +\infty$ and consider the Diophantine space $(\ell^p(\R),\Q^\infty,H_\std)$. For any function $\psi\to 0$, $\BA_\psi$ is prevalent.
\end{theorem}
\begin{proof}
Let $(\radius_n)_1^\infty$ be the unique sequence satisfying $\radius_0 = 1$ and
\[
\radius_{n + 1} = \frac{\radius_n}{2^{n + 5}}\cdot
\]
For each $n\in\N$, let $N_n\in\N$ be large enough so that
\[
\psi(q) \leq \frac{\radius_{n + 1}}{8} \all q\geq N_n.
\]
Let
\[
M_n = (2^{n + 1}N_n!)^p
\]
and let
\[
\vv_n = \frac{\radius_n}{M_n^{1/p}}\sum_{i = 1}^{M_n}\ee_i.
\]
\begin{claim}
\label{claimatmostone}
For each $n\in\N$ and for each $\xx\in\ell^p(\R)$,
\[
\#\left\{i = 0,\ldots,2^n - 1:B(\xx + i\vv_n,\radius_n/4)\cap \frac{\Z^\infty}{N_n!}\neq \emptyset\right\} \leq 1.
\]
\end{claim}
\begin{subproof}
By contradiction, suppose that we have $0 \leq i_1 < i_2 < 2^n$ and $\xx_1,\xx_2\in\ell^p(\R)$ such that
\[
\xx_j\in B(\xx + i_j\vv_n,\radius_n/4)\cap \frac{\Z^\infty}{N_n!},\;\; j = 1,2.
\]
Thus
\begin{align*}
\frac{\radius_n}{2}
&\geq \left\|(\xx_2 - \xx_1) - \big[(\xx + i_2\vv_n) - (\xx + i_1\vv_n)\big]\right\|_p\\
&= \|(\xx_2 - \xx_1) - (i_2 - i_1)\vv_n\|_p\\
&\geq \dist\left((i_2 - i_1)\vv_n,\frac{\Z^\infty}{N_n!}\right).
\end{align*}
But $(i_2 - i_1)\vv_n$ lies in the cube $[-1/(2N_n!),1/(2N_n!)]^\infty$, which implies that
\[
\dist\left((i_2 - i_1)\vv_n,\frac{\Z^\infty}{N_n!}\right) = \|(i_2 - i_1)\vv_n\|_p \geq \|\vv_n\|_p = \radius_n,
\]
a contradiction.
\end{subproof}
Let
\[
\mu_n = \frac{1}{2^n}\sum_{i = 0}^{2^n - 1}\delta_{i\vv_n},
\]
and let $\mu$ be the infinite convolution of the $\mu_n$s. This convolution converges since for each $n$, $\mu_n$ is supported on $B(\0,\radius_n/4)$ and the sequence $(\radius_n)_1^\infty$ is summable.

Since the $\mu_n$s are compactly supported, their infinite convolution $\mu$ is also compactly supported. To show that $\BA_\psi$ is prevalent, it suffices to show that $\mu(\WA_\psi + \vv) = 0$ for any vector $\vv\in\ell^p(\R)$.

Indeed, fix such a vector $\vv$. Then a $\mu$-random point $\xx\in\ell^p(\R)$ can be represented in the form
\[
\xx = \sum_{n = 0}^\infty \xx_n,
\]
where each point $\xx_n$ is $\mu_n$-random, i.e.
\[
\xx_n = i_n\vv_n,
\]
where $i_n\in\{0,\ldots,2^n - 1\}$ is random with respect to the uniform distribution. Fix $n\in\N$; applying Claim \ref{claimatmostone} with $\xx = \sum_{i = 0}^{n - 1}\xx_i - \vv$ we have
\[
\prob\left(B\left(\sum_{m = 0}^n \xx_m - \vv,\radius_n/4\right)\cap \frac{\Z^\infty}{N_n!}\neq \emptyset\right) \leq \frac{1}{2^n}\cdot
\]
Thus by the easy direction of Borel--Cantelli we have
\[
\prob\left(\exists^\infty n\in\N \;\; B\left(\sum_{m = 0}^n \xx_m - \vv,\radius_n/4\right)\cap \frac{\Z^\infty}{N_n!}\neq \emptyset\right) = 0.
\]
Thus to complete the proof it suffices to show the following:
\begin{claim}
If a sequence $(\xx_n)_1^\infty$ satisfies
\[
\xx = \sum_{n = 0}^\infty \xx_n \in \WA_\psi + \vv,
\]
then there exist infinitely many $n\in\N$ for which
\[
B\left(\sum_{m = 0}^n \xx_m - \vv,\radius_n/4\right)\cap \frac{\Z^\infty}{N_n!}\neq \emptyset.
\]
\end{claim}
\begin{proof}
Let $(\rr_k)_1^\infty$ be a sequence of rational points whose limit is $\xx - \vv$ and which satisfy
\[
\|\xx - \vv - \rr_k\|_p \leq \psi(q_k),
\]
where
\[
q_k = H_\std(\rr_k).
\]
Fix $k\in\N$, and let $n$ be minimal so that $N_n\geq q_k$. Then $N_{n - 1} < q_k$, and so
\[
\psi(q_k) \leq \frac{\radius_n}{8}\cdot
\]
On the other hand,
\[
\left\|\sum_{m = n + 1}^\infty\xx_m\right\|_p \leq \sum_{m = n + 1}^\infty 2^m\radius_m \leq \frac{\radius_n}{8}\cdot
\]
Thus,
\[
\left\|\sum_{m = 0}^n\xx_m - \vv - \rr_k\right\|_p \leq \frac{\radius_n}{4},
\]
i.e.
\[
\rr_k\in B\left(\sum_{m = 0}^n \xx_m - \vv,\radius_n/4\right)\cap \frac{\Z^\infty}{N_n!}\cdot
\]
\QEDmod\end{proof}
\end{proof}

As a consequence of Theorem \ref{theoremLpkhinchin}, we easily deduce the following:
\begin{theorem}
\label{theoremnoBA}
Fix $1\leq p < +\infty$ and consider the Diophantine space $(\ell^p(\R),\Q^\infty,H_\std)$. Then the function $\psi(q) = 1$ is an optimal Dirichlet function but $\BA_\psi = \emptyset$.
\end{theorem}
\begin{proof}
The fact that $\psi\equiv 1$ is a Dirichlet function which has no badly approximable points is true of every Diophantine space, simply because $\QQ$ is dense in $X$. It remains to show optimality. Let $\phi:\Rplus\to(0,+\infty)$ be a function such that $\phi(q)\tendsto q 0$. Then $\sqrt\phi(q)\tendsto q 0$, and so by Theorem \ref{theoremLpkhinchin}, the set $\BA_{\sqrt\phi}$ is prevalent and in particular nonempty. Then by Theorem \ref{theoremBAimpliesoptimality}, the function $\phi$ cannot be a Dirichlet function, since $\frac{\phi(q)}{\sqrt\phi(q)}\tendsto q 0$.
\end{proof}

\subsection{Proof of Theorem \ref{theoremLpjarnikbesicovitch}}
\begin{theorem}
\label{theoremLpjarnikbesicovitch}
Fix $1\leq p < +\infty$ and consider the Diophantine space $(\ell^p(\R),\Q^\infty,H_\std)$. For any function $\psi\to 0$,
\[
\HD(\WA_\psi) = +\infty.
\]
\end{theorem}

[Add in proof.]

\section{Diophantine approximation in $c_0(\R)$}

\begin{theorem}
\label{theoremc0khinchin}
Consider the Diophantine space $(c_0(\R),\Q^\infty,H_\std)$.The function $\psi(q) = 1/q$ is an optimal Dirichlet function, and $\BA_\psi$ is prevalent.
\end{theorem}
\begin{proof}
For any $\xx = (x_n)_1^\infty\in c_0(\R)$, fix $q\in\N$ and for each $n\in\N$ let
\[
r_n^{(q)} = \frac{1}{q}\lfloor qx_n\rfloor.
\]
Then $\rr_q = (r_n^{(q)})_1^\infty\in\Q^\infty$ and $q\rr_q\in\Z^\infty$, which implies $H_\std(\rr_q)\leq q$. On the other hand,
\[
\|\xx - \rr_q\|_\infty = \sup_{n \in\N} \frac{1}{q}|qx_n - \lfloor qx_n\rfloor| \leq \frac{1}{q}\cdot 
\]
Clearly $\rr_q\tendsto q \xx$. This demonstrates that $\psi(q) = 1/q$ is a Dirichlet function.

To show that $\BA_\psi$ is prevalent, define the measure $\mu$ on $c_0(\R)$ to be the image of Lebesgue measure on $[0,1]^\infty$ under multiplication by the sequence $(1/n)_1^\infty$. Fix $\vv\in c_0(\R)$, and we will show that $\mu(\WA_\psi + \vv) = 0$. Indeed, for each $q\in\N$,
\begin{align*}
\prob\big(\dist(q[\xx - \vv],\Z^\infty)\leq 1/4\big)
&= \prod_{n = 1}^\infty\prob\big(\dist(q[x_n - v_n],\Z)\leq 1/4\big)\\
&\leq (2/3)^{\#(n\in\N:n\leq q)}
= (2/3)^q,
\end{align*}
since $\prob\big(\dist(q[x_n - v_n],\Z)\leq 1/4\big) \leq 2/3$ whenever $n\leq q$. Thus by the easy direction of the Borel--Cantelli lemma, we have
\[
\prob\left(\exists^\infty q\in\N \;\; \dist(q[\xx - \vv],\Z^\infty)\leq 1/4\right) = 0.
\]
But clearly, if $\xx\in\WA_\psi + \vv$, then $\dist(q[\xx - \vv],\Z^\infty)\leq 1/4$ for infinitely many $q\in\N$. Thus $\mu(\WA_\psi + \vv) = 0$, and $\BA_\psi$ is prevalent.
\end{proof}

\begin{theorem}
\label{theoremc0jarnikbesicovitch}
Consider the Diophantine space $(c_0(\R),\Q^\infty,H_\std)$. For any function $\psi\to 0$,
\[
\HD(\WA_\psi) = +\infty.
\]
\end{theorem}
[Add in proof.]

}

\draftnewpage
\section{Infinite-dimensional non-cobounded case}
\label{sectionnoncobounded}
In this section, we assume that $\Lambda$ is a non-cobounded lattice in an infinite-dimensional Banach space $X$, and we consider the Diophantine space $(X,\Q\Lambda,H_\std)$. We begin by proving the following:

\begin{theorem}[Khinchin-type theorem, Jarn\'ik--Schmidt type theorem]
\label{theoremkhinchinnoncobounded}
For any function $\psi\to 0$, $\BA_\psi$ is prevalent. In particular, $\HD(\BA_\psi) = +\infty$.
\end{theorem}
\begin{proof}
We will need the following lemma:
\begin{lemma}
\label{lemmaw}
For any $0 < \varepsilon < R < +\infty$, there exists $\ww\in X$ so that $\|\ww\| = R$ and
\[
\dist(\ww,\Lambda) \geq R - \varepsilon.
\]
\end{lemma}
\begin{subproof}
Since $\Lambda$ is not cobounded, there exists $\xx\in X$ such that $S := \dist(\xx,\Lambda) \geq R$. By the definition of distance, there exists $\pp\in\Lambda$ such that
\begin{equation}
\label{xpbounds}
S \leq \|\xx - \pp\| \leq S + \varepsilon.
\end{equation}
Let
\[
\ww = R\frac{\xx - \pp}{\|\xx - \pp\|}.
\]
Clearly $\|\ww\| = R$. On the other hand,
\begin{align*}
\dist(\ww,\Lambda) &\geq \dist\left(\xx - \pp,\Lambda \right) - \dist(\ww,\xx - \pp)\\
&= S - \|\xx - \pp\|\left|1 - \frac{R}{\|\xx - \pp\|}\right|\\
&= S - \big|\|\xx - \pp\| - R\big|\\
&\geq S - \left|S - R\right| - \varepsilon \by{\eqref{xpbounds}}\\
&= R - \varepsilon. \since{$S\geq R$}
\end{align*}
\end{subproof}
Let $(\radius_n)_1^\infty$ be the unique sequence satisfying $\radius_1 = 1$ and
\begin{equation}
\label{rhodef}
\radius_{n + 1} = \frac{\radius_n}{2^{n + 5}}\cdot
\end{equation}
For each $n\in\N$, let $N_n\in\N$ be large enough so that
\begin{equation}
\label{Nndef}
\psi(q) \leq \frac{\radius_{n + 1}}{8} \all q\geq N_n.
\end{equation}
Let
\[
M_n = 2^n N_n!
\]
By Lemma \ref{lemmaw}, there exists $\ww_n\in X$ be such that $\|\ww_n\| = M_n$ and
\begin{equation}
\label{Mnrn}
\dist(\ww_n,\Lambda)\geq M_n - \radius_n/4.
\end{equation}
Let $\vv_n = \frac{\radius_n}{M_n}\ww_n$, so that $\|\vv_n\| = \radius_n$.
\begin{claim}
\label{claimatmostone}
For each $n\in\N$ and for each $\xx\in X$,
\[
\#\left\{i = 0,\ldots,2^n - 1:B(\xx + i\vv_n,\radius_n/4)\cap \frac{\Lambda}{N_n!}\neq \emptyset\right\} \leq 1.
\]
\end{claim}
\begin{proof}
By contradiction, suppose there exist $0 \leq i_1 < i_2 < 2^n$ and $\xx_1,\xx_2\in X$ such that
\[
\xx_j\in B(\xx + i_j\vv_n,\radius_n/4)\cap \frac{\Lambda}{N_n!},\;\; j = 1,2.
\]
Thus
\begin{align*}
\frac{\radius_n}{2}
&\geq \left\|(\xx_2 - \xx - i_2\vv_n) - (\xx_1 - \xx - i_1\vv_n)\right\|\\
&= \|(\xx_2 - \xx_1) - (i_2 - i_1)\vv_n\|\\
&\geq \dist\left((i_2 - i_1)\vv_n,\frac{\Lambda}{N_n!}\right)\\
&= \frac{1}{N_n!}\dist\big((i_2 - i_1)N_n!\vv_n,\Lambda\big)\\
&= \frac{1}{N_n!}\dist\left((i_2 - i_1)N_n!\frac{\radius_n}{M_n}\ww_n,\Lambda\right).
\end{align*}
Now
\[
(i_2 - i_1)N_n!\frac{\radius_n}{M_n} = (i_2 - i_1)\frac{\radius_n}{2^n} \leq \radius_n \leq 1,
\]
and so by the triangle inequality
\begin{align*}
N_n!\frac{\radius_n}{2} &\geq \dist\left((i_2 - i_1)N_n!\frac{\radius_n}{M_n}\ww_n,\Lambda\right)\\
&\geq \dist(\ww_n,\Lambda) - \left\|\ww_n - (i_2 - i_1)N_n!\frac{\radius_n}{M_n}\ww_n\right\| \\
&\geq (M_n - \radius_n/4) - (M_n - (i_2 - i_1)N_n! \radius_n) \by{\eqref{Mnrn}}\\
&\geq (N_n! - 1/4)\radius_n, \since{$i_2 - i_1 \geq 1$}
\end{align*}
a contradiction.
\QEDmod\end{proof}
For each $n\in\N$, let
\[
\mu_n = \frac{1}{2^n}\sum_{i = 0}^{2^n - 1}\delta_{i\vv_n};
\]
then $\mu_n$ is a compactly supported probability measure on $B(\0,2^n\radius_n)$. Define
\[
\Sigma:\prod_{n = 1}^\infty B(\0,2^n\radius_n)\to X
\]
by
\begin{equation}
\label{Sigmadef}
\Sigma((\xx_n)_1^\infty) = \sum_{n = 1}^\infty \xx_n,
\end{equation}
and let $\mu = \Sigma[\prod_{n = 1}^\infty \mu_n]$. Note that if $K_n$ is the support of $\mu_n$, then $\mu$ gives full measure to $\Sigma(\prod_{n = 1}^\infty K_n)$, which is compact, so $\mu$ is compactly supported.

To complete the proof, we will show that $\mu$ is transverse to $\WA_\psi$. To this end, fix $\vv\in X$, and we will show that $\mu(\WA_\psi + \vv) = 0$.

\ignore{
A $\mu$-random point $\xx\in X$ can be represented in the form
\[
\xx = \Sigma((\xx_n)_1^\infty) = \sum_{n = 0}^\infty \xx_n,
\]
where each point $\xx_n$ is $\mu_n$-random, i.e.
\[
\xx_n = i_n\vv_n,
\]
where $i_n\in\{0,\ldots,2^n - 1\}$ is random with respect to the uniform distribution.}

Fix $n\in\N$; for each sequence $(\xx_j)_1^{n - 1}$, applying Claim \ref{claimatmostone} with $\xx = \sum_{j = 1}^{n - 1}\xx_j - \vv$ shows that
\[
\mu_n\left\{\xx_n: B\left(\sum_{j = 0}^n \xx_j - \vv,\radius_n/4\right)\cap \frac{\Lambda}{N_n!}\neq \emptyset\right\} \leq \frac{1}{2^n},
\]
and Fubini's theorem gives
\[
\left(\prod_{j = 1}^\infty\mu_j\right)\left\{(\xx_j)_1^\infty: B\left(\sum_{j = 0}^n \xx_j - \vv,\radius_n/4\right)\cap \frac{\Lambda}{N_n!}\neq \emptyset\right\} \leq \frac{1}{2^n}\cdot
\]
Thus by the easy direction of the Borel--Cantelli lemma, the set
\[
N =\left\{(\xx_j)_1^\infty:\exists^\infty n\in\N \;\; B\left(\sum_{j = 0}^n \xx_j - \vv,\radius_n/4\right)\cap \frac{\Lambda}{N_n!}\neq \emptyset\right\}
\]
is a $\prod_{j = 1}^\infty\mu_j$-nullset. So to complete the proof, it suffices to show the following:
\begin{claim}
\label{claimNnfactorial}
$\Sigma^{-1}(\WA_\psi + \vv) \subset N$, i.e. if a sequence $(\xx_n)_1^\infty \in \prod_{n = 1}^\infty B(\0,2^n\radius_n)$ satisfies
\[
\xx = \sum_{n = 0}^\infty \xx_n \in \WA_\psi + \vv,
\]
then there exist infinitely many $n\in\N$ for which
\begin{equation}
\label{rn4LambdaNn}
B\left(\sum_{j = 0}^n \xx_j - \vv,\radius_n/4\right)\cap \frac{\Lambda}{N_n!}\neq \emptyset.
\end{equation}
\end{claim}
\begin{proof}
Let $(\rr_k)_1^\infty$ be a sequence of rational points whose limit is $\xx - \vv$ and which satisfy
\[
\|\xx - \vv - \rr_k\| \leq \psi(q_k),
\]
where $q_k = H_\std(\rr_k)$.

Fix $k\in\N$, and let $n = n_k$ be minimal so that $N_n\geq q_k$. Then $N_{n - 1} < q_k$, and so by \eqref{Nndef},
\[
\psi(q_k) \leq \frac{\radius_n}{8}\cdot
\]
On the other hand, by \eqref{rhodef},
\[
\left\|\sum_{j = n + 1}^\infty\xx_j\right\| \leq \sum_{j = n + 1}^\infty 2^j\radius_j \leq \frac{\radius_n}{8}\cdot
\]
Combining the three preceding equations gives
\[
\left\|\sum_{j = 0}^n\xx_j - \vv - \rr_k\right\| \leq \frac{\radius_n}{4},
\]
i.e.
\[
\rr_k\in B\left(\sum_{j = 0}^n \xx_j - \vv,\radius_n/4\right)\cap \frac{\Lambda}{N_n!}\cdot
\]
Since the sequence $(n_k)_1^\infty$ is clearly unbounded, this demonstrates that \eqref{rn4LambdaNn} holds for infinitely many $n$.
\QEDmod\end{proof}
Thus $\mu$ is transverse to $\WA_\psi$ and so $\BA_\psi$ is prevalent; thus $\HD(\BA_\psi) = +\infty$ by Proposition \ref{propositionHDinfinity}.
\end{proof}

Next, we deduce Theorem \ref{theoremdirichletnoncobounded} as a corollary of Theorem \ref{theoremkhinchinnoncobounded}.

\begin{theorem}[Dirichlet-type theorem]
\label{theoremdirichletnoncobounded}
The function $\psi_0 \equiv 1$ is an optimal uniformly Dirichlet function. However, $\BA_{\psi_0} = \emptyset$.
\end{theorem}
\begin{proof}
The fact that $\psi\equiv 1$ is a Dirichlet function which has no badly approximable points is true of every Diophantine space, simply because $\QQ$ is dense in $X$. It remains to show optimality. Let $\phi:\Rplus\to(0,+\infty)$ be a function such that $\frac{\phi}{\psi} = \phi \to 0$. Then $\sqrt\phi \to 0$, and so by Theorem \ref{theoremkhinchinnoncobounded}, the set $\BA_{\sqrt\phi}$ is prevalent and in particular nonempty. Then by Theorem \ref{theoremBAimpliesoptimality}, the function $\phi$ cannot be a Dirichlet function, since $\frac{\phi}{\sqrt\phi}\to 0$.
\end{proof}

Finally, we prove the infinite-dimensional version of the Jarn\'ik--Besicovitch theorem.

\begin{theorem}[Jarn\'ik--Besicovitch type theorem]
\label{theoremjarniknoncobounded}
~
\begin{itemize}
\item[(i)] For any $s\geq 0$, we have $\HD(\WA_{\psi_s}) = +\infty$.
\item[(ii)] Suppose that $\Lambda$ is not strongly discrete. Then for any nonincreasing function $\psi\to 0$, $\HD(\WA_\psi) = +\infty$. In fact, for any nondecreasing function $f:(0,+\infty)\to(0,+\infty)$, $\HH^f(\WA_\psi) = +\infty$.
\end{itemize}
\end{theorem}
\begin{remark}
In this theorem, the hypothesis that $\Lambda$ is not cobounded is not used; cf. Theorem \ref{theoremjarnikcobounded}.
\end{remark}
~
\begin{proof}[Proof of Theorem \ref{theoremjarniknoncobounded}]
~
\begin{itemize}
\item[(i)]
For each $d\in\N$, let $X_d\subset X$ be a subspace of dimension $d$ such that $X_d\cap\Lambda$ is a lattice in $X_d$. Then by Theorem \ref{theoremjarnik},
\[
\HD(\WA_{\psi_s}(X_d,\Q\Lambda\cap X_d,H_\std)) = \frac{d + 1}{s}\cdot
\]
But clearly $\WA_{\psi_s}(X,\Q\Lambda,H_\std)\supset\WA_{\psi_s}(X_d,\Q\Lambda\cap X_d,H_\std)$, whence
\[
\HD(\WA_{\psi_s}(X,\Q\Lambda,H_\std)) \geq \frac{d + 1}{s} \tendsto d +\infty.
\]
\item[(ii)]
Let $\psi\to 0$ be a nonincreasing function, and let $f:(0,+\infty)\to(0,+\infty)$ be a nondecreasing function. Let $\varepsilon_\Lambda = \min_{\pp\in\Lambda\butnot\{\0\}}\|\pp\| > 0$, and let $C_\Lambda > 0$ be large enough so that $\#(\Lambda\cap B(\0,C_\Lambda)) = +\infty$. Choose a sequence $(q_n)_1^\infty$ by induction as follows: Let $q_0 = 1$, and if $q_n$ has been chosen, let $q_{n + 1} \in q_n\N\butnot\{q_n\}$ be large enough so that $2C_\Lambda/q_{n + 1} \leq \min(\psi(q_n)/n,\varepsilon_\Lambda/(3q_n))$. Let $S = \Lambda\cap B(\0,C_\Lambda)$, and define $\pi:S^\N\to X$ by
\[
\pi((\pp_n)_1^\infty) = \sum_{n\in\N}\frac{\pp_n}{q_n}\cdot
\]
\begin{claim}
$\pi(S^\N)\subset\WA_\psi$.
\end{claim}
\begin{subproof}
Fix $(\pp_n)_1^\infty\in S^\N$, and let $\xx = \pi((\pp_n)_1^\infty) = \sum_{n\in\N}\pp_n/q_n$. Then for each $N\in\N$
\[
\left\|\xx - \sum_{n\leq N}\frac{\pp_n}{q_n}\right\| \leq \sum_{n > N}\frac{\|\pp_n\|}{q_n} \leq C_\Lambda\sum_{n > N}\frac{1}{q_n} \leq \frac{2C_\Lambda}{q_{N + 1}} \leq \psi(q_N)/N.
\]
On the other hand, since $q_1\divides q_2\divides \cdots \divides q_N$, we have
\[
H_\std\left(\sum_{n\leq N}\frac{\pp_n}{q_n}\right) \leq q_N,
\]
and so since $\psi$ is nonincreasing, we have
\[
\left\|\xx - \sum_{n\leq N}\frac{\pp_n}{q_n}\right\| \leq \frac{1}{n}\psi\circ H_\std\left(\sum_{n\leq N}\frac{\pp_n}{q_n}\right)
\]
and thus $\xx\in\WA_\psi$.
\end{subproof}
\ignore{
\begin{claim}
$\pi$ is bi-Lipschitz if $S^\N$ is given the metric
\[
\dist\left((\pp_n^{(1)})_1^\infty,(\pp_n^{(2)})_1^\infty\right) = 1/q_{\min\{n:\pp_n^{(1)}\neq\pp_n^{(2)}\}}.
\]
\end{claim}
\begin{proof}
Let $N = \min\{n:\pp_n^{(1)}\neq\pp_n^{(2)}\}$. Then
\begin{align*}
\left\|\pi((\pp_n^{(1)})_1^\infty) - \pi((\pp_n^{(1)})_1^\infty)\right\|
&= \left\|\sum_{n\geq N}\frac{\pp_n^{(1)} - \pp_n^{(2)}}{q_n}\right\|\\
&\geq \frac{\|\pp_N^{(1)} - \pp_N^{(2)}\|}{q_N} - \sum_{n > N} \frac{\|\pp_n^{(1)} - \pp_n^{(2)}\|}{q_n}\\
&\geq \frac{\varepsilon_\Lambda}{q_N} - \sum_{n > N} \frac{2C_\Lambda}{q_n}\\
&\geq \frac{\varepsilon_\Lambda}{q_N} - \frac{4C_\Lambda}{q_{N + 1}}\\
&\geq \frac{\varepsilon_\Lambda}{q_N} - \frac{2\varepsilon_\Lambda}{3q_N}\\
&= \frac{\varepsilon_\Lambda}{3q_N} \asymp_\times \frac{1}{q_N} = \dist\left((\pp_n^{(1)})_1^\infty,(\pp_n^{(2)})_1^\infty\right);
\end{align*}
the other direction is similar.
\end{proof}
}

\begin{claim}
If $\CC$ is any collection of subsets of $X$ of diameter less than $\varepsilon_\Lambda/3$ which covers $\pi(S^\N)$, then $\CC$ contains an infinite collection of sets whose diameters are bounded from below.
\end{claim}
\begin{subproof}
By contradiction, suppose not; then for each $n$, the set
\[
\CC_n := \{A\in\CC: \varepsilon_\Lambda/(3q_{n + 1}) \leq \diam(A) < \varepsilon_\Lambda/(3q_n)\}
\]
is finite. We now choose a sequence $(\pp_n)_1^\infty$ in $S^\N$ by induction. If $\pp_1,\ldots,\pp_{N - 1}$ have been chosen, then for each $\pp\in S$ let
\[
C_{N,\pp} = \sum_{n < N}\frac{\pp_n}{q_n} + \frac{\pp}{q_N} + B(\0,\varepsilon_\Lambda/(3q_N)).
\]
The sets $(C_{N,\pp})_{\pp\in S}$ are disjoint; in fact, the distance between $C_{N,\pp}$ and $C_{N,\w\pp}$ for $\pp\neq\w\pp$ is always at least $\varepsilon_\Lambda/(3q_N)$. Thus each $A\in\CC_N$ can intersect at most one of the sets $C_{N,\pp}$, so since $\#(S) = +\infty$ there exists $\pp_N\in S$ such that $C_{N,\pp_N}$ is disjoint from $\bigcup(\CC_N)$. This completes the inductive step.

Calculation (based on the inequality $2C_\Lambda/q_{N + 1} \leq \varepsilon_\Lambda/(3q_N)$) shows that the point $\xx = \pi((\pp_n)_1^\infty)$ is in each of the sets $C_{N,\pp_N}$, and so it is not in any of the sets $\bigcup(\CC_N)$. This contradicts that $\CC$ covers $\pi(S^\N)$.
\end{subproof}
Now if $f:(0,+\infty)\to(0,+\infty)$ is nondecreasing, then the equation $\HH^f(\pi(S^\N)) = +\infty$ is evident from the claim. Finally, setting $f(t) = t^s$ with $s$ arbitrary shows that $\HD(\WA_\psi) = +\infty$.
\end{itemize}
\end{proof}

\draftnewpage
\section{Infinite-dimensional cobounded case}
\label{sectioncobounded}
In this section, we assume that $\Lambda$ is a cobounded lattice in an infinite-dimensional Banach space $X$, and we consider the Diophantine space $(X,\Q\Lambda,H_\std)$. We begin by proving the following:

\begin{theorem}[Khinchin-type theorem, Jarn\'ik--Schmidt type theorem]
\label{theoremkhinchincobounded}
The set $\BA_{\psi_1}$ is prevalent. In particular, $\HD(\BA_{\psi_1}) = +\infty$.
\end{theorem}

The proof will follow the same lines as the proof of Theorem \ref{theoremkhinchinnoncobounded}, but with some modifications.

\ignore{

\begin{theorem}
Then the Dirichlet function $\psi(q) = 1/q$ is optimal; indeed, $\BA_\psi$ is hyperplane winning.
\end{theorem}
The proof is a straightforward generalization of the finite-dimensional proof, based on the
\begin{lemma}[Simplex Lemma for Infinite-Dimensional Banach Spaces]
\label{lemmabanachsimplex}
There exists $\varepsilon > 0$ so that for all $Q\in\N$ and for every ball $B = B(\xx,\varepsilon/Q)\subset X$, the set of rational points $\rr\in B(\xx,\varepsilon/Q)$ satisfying $H_\std(\rr)\leq Q$ has cardinality at most $Q$, and in particular is contained in an affine subspace of $X$ of codimension one.
\end{lemma}
\begin{proof}
Let $\varepsilon > 0$ be small enough so that $B(\0,\varepsilon)\subset\Delta$ (this is possible since $\Lambda$ is discrete). For each $q\leq Q$, we have
\begin{align*}
\#\left(B\cap \frac{\Lambda}{q}\right)
&= \#(B(q\xx,\varepsilon q/Q)\cap\Lambda)\\
&\leq \#(B(q\xx,\varepsilon)\cap \Lambda \leq 1.
\end{align*}
Thus the set of rational points $\rr\in B(\xx,\varepsilon/Q)$ satisfying $H_\std(\rr)\leq Q$ has cardinality at most $Q$.
\end{proof}

}
\begin{proof}[Proof of Theorem \ref{theoremkhinchincobounded}]
\begin{lemma}
There exists a sequence of unit vectors $(\ee_i)_1^\infty$ satisfying
\begin{equation}
\label{separated}
\|\ee_j - \ee_i\|\geq 1 \text{ whenever } i\neq j.
\end{equation}
\end{lemma}
\begin{subproof}
We construct the sequence $(\ee_i)_1^\infty$ by induction. Suppose that $(\ee_i)_1^{n - 1}$ have been defined, and let $V = \sum_{i = 1}^{n - 1} \R\ee_i$. Let $\ww$ be a unit vector in $X/V$, and let $\ee_n\in X$ be a unit vector representing $\ww$. Then for all $i < n$,
\[
\|\ee_n - \ee_i\| \geq \dist(\ee_n,V) = \|\ww\| = 1.
\]
This demonstrates \eqref{separated}.
\end{subproof}
Let $\varepsilon_\Lambda = \min_{\pp\in\Lambda}\|\pp\| > 0$, let $\lambda = 16$, and for each $n\in\N$ and $i = 1,\ldots,\lambda^{2n}$ let
\[
\vv_{i,n} = \frac{\varepsilon_\Lambda\ee_i}{4\lambda^n}\cdot
\]
\begin{claim}
\label{claimq2n}
For any point $\xx\in X$,
\begin{equation}
\label{q2n}
\#\left\{i = 1,\ldots,\lambda^{2n}: B\left(\xx + \vv_{i,n}, \frac{\varepsilon_\Lambda}{16\lambda^n}\right)\cap \left(\bigcup_{q = 1}^{\lambda^n}\frac{\Lambda}{q}\right) \neq \emptyset\right\} \leq \lambda^n.
\end{equation}
\end{claim}
\begin{subproof}
\ignore{
\begin{align*}
&\mu_n\left(\bigcup_{q = 1}^{\lambda^n}B\left(\xx + \frac{\Lambda}{q},\frac{\varepsilon_\Lambda}{16\lambda^n}\right)\right)\\
&= \frac{1}{\lambda^{2n}}\#\left\{i = 1,\ldots,\lambda^{2n}: \vv_{i,n} \in \bigcup_{q = 1}^{\lambda^n} B\left(\xx + \frac{\Lambda}{q},\frac{\varepsilon_\Lambda}{16\lambda^n}\right)\right\}\noreason\\
&\leq \frac{1}{\lambda^{2n}} \sum_{q = 1}^{\lambda^n}\#\left\{i = 1,\ldots,\lambda^{2n}: \vv_{i,n} \in B\left(\xx + \frac{\Lambda}{q},\frac{\varepsilon_\Lambda}{16\lambda^n}\right)\right\}.
\end{align*}}
Fix $q = 1,\ldots,\lambda^n$, and by contradiction suppose there exist $1\leq i_1 < i_2 \leq \lambda^{2n}$ such that $B\left(\xx + \vv_{i,n}, \frac{\varepsilon_\Lambda}{16\lambda^n}\right)\cap \frac{\Lambda}{q} \neq \emptyset$. Then there exist $\pp_1,\pp_2\in\Lambda$ so that
\[
\left\|\vv_{i_j,n} + \xx - \frac{\pp_j}{q}\right\| \leq \frac{\varepsilon_\Lambda}{16\lambda^n},
\]
which implies that
\[
\|\vv_{i_1,n} - \vv_{i_2,n}\| - \frac{\varepsilon_\Lambda}{8\lambda^n} \leq \frac{1}{q}\|\pp_1 - \pp_2\| \leq \|\vv_{i_1,n} - \vv_{i_2,n}\| + \frac{\varepsilon_\Lambda}{8\lambda^n}\cdot
\]
By \eqref{separated} we have
\[
\frac{\varepsilon_\Lambda}{4\lambda^n} \leq \|\vv_{i_1,n} - \vv_{i_2,n}\| \leq \frac{\varepsilon_\Lambda}{2\lambda^n},
\]
and so
\[
\frac{\varepsilon_\Lambda}{8\lambda^n} \leq \frac{1}{q}\|\pp_1 - \pp_2\| \leq \frac{5\varepsilon_\Lambda}{8\lambda^n}.
\]
The lower bound implies that $\pp_1 - \pp_2 \neq \0$, so since $\pp_1 - \pp_2 \in \Lambda$ we have $\|\pp_1 - \pp_2\| \geq \varepsilon_\Lambda$; thus
\[
\frac{5\varepsilon_\Lambda}{8\lambda^n} \geq \frac{\varepsilon_\Lambda}{q} \geq \frac{\varepsilon_\Lambda}{\lambda^n},
\]
a contradiction. Thus
\[
\#\left\{i = 1,\ldots,\lambda^{2n}: B\left(\vv_{i,n} + \xx,\frac{\varepsilon_\Lambda}{16\lambda^n}\right)\cap\frac{\Lambda}{q}\neq\emptyset\right\} \leq 1,
\]
and summing over $q = 1,\ldots,\lambda^n$ yields \eqref{q2n}.
\end{subproof}
At this point, the proof follows much the same structure as the proof of Theorem \ref{theoremkhinchinnoncobounded}. For each $n\in\N$, let
\[
\mu_n = \frac{1}{\lambda^{2n}}\sum_{i = 1}^{\lambda^{2n}} \delta_{\vv_{i,n}};
\]
then $\mu_n$ is a compactly supported probability measure on $B(\0,\varepsilon_\Lambda/(4\lambda^n))$. Let
\[
\Sigma:\prod_{n = 1}^\infty B(\0,\varepsilon_\Lambda/(4\lambda^n)) \to X
\]
be defined by \eqref{Sigmadef}, and let $\mu = \Sigma\left[\prod_{n = 1}^\infty \mu_n\right]$. As in the proof of Theorem \ref{theoremkhinchinnoncobounded}, $\mu$ is compactly supported; fix $\vv\in X$, and we will show that $\mu(\WA_\psi + \vv) = 0$.

Fix $n\in\N$; for each sequence $(\xx_j)_1^{n - 1}$, applying Claim \ref{claimatmostone} with $\xx = \sum_{j = 1}^{n - 1}\xx_j - \vv$ shows that
\[
\mu_n\left\{\xx_n:  B\left(\sum_{j = 0}^n \xx_j - \vv,\varepsilon_\Lambda/16\lambda^n\right)\cap \left(\bigcup_{q = 1}^{\lambda^n}\frac{\Lambda}{q}\right)\neq \emptyset\right\} \leq \frac{1}{\lambda^n},
\]
and Fubini and Borel--Cantelli imply that
\[
N = \left\{(\xx_j)_1^\infty:\exists^\infty n\in\N \;\; B\left(\sum_{j = 0}^n \xx_j - \vv,\varepsilon_\Lambda/16\lambda^n\right)\cap \left(\bigcup_{q = 1}^{\lambda^n}\frac{\Lambda}{q}\right)\neq \emptyset\right\}
\]
is a $\prod_{j = 1}^\infty\mu_j$-nullset. So to complete the proof, it suffices to show the following:
\begin{claim}
\label{claimNnfactorialmod}
$\Sigma^{-1}(\WA_\psi + \vv) \subset N$, i.e. if a sequence $(\xx_n)_1^\infty\in\prod_{n = 1}^\infty B(\0,\varepsilon_\Lambda/(4\lambda^n))$ satisfies
\[
\xx = \sum_{n = 0}^\infty \xx_n \in \WA_\psi + \vv,
\]
then there exist infinitely many $n\in\N$ for which
\[
B\left(\sum_{j = 0}^n \xx_j - \vv,\frac{\varepsilon_\Lambda}{16\lambda^n}\right)\cap \left(\bigcup_{q = 1}^{\lambda^n}\frac{\Lambda}{q}\right)\neq \emptyset.
\]
\end{claim}
\begin{subproof}
Let $(\rr_k)_1^\infty$ be a sequence of rational points whose limit is $\xx - \vv$ and which satisfy
\[
\|\xx - \vv - \rr_k\| \leq \psi(q_k)/k,
\]
where $q_k = H_\std(\rr_k)$.

Fix $k\in\N$, and let $n = n_k$ be minimal so that $\lambda^n\geq q_k$. Then $\lambda^{n - 1} < q_k$, and so,
\[
\psi(q_k) = \frac{1}{q_k} \leq \frac{1}{\lambda^{n - 1}}\cdot
\]
On the other hand, by \eqref{rhodef},
\[
\left\|\sum_{j = n + 1}^\infty\xx_j\right\| \leq \sum_{j = n + 1}^\infty \frac{\varepsilon_\Lambda}{4\lambda^j} \leq \frac{\varepsilon_\Lambda}{2\lambda^{n + 1}}\cdot
\]
Combining the three preceding equations gives
\[
\left\|\sum_{j = 0}^n\xx_j - \vv - \rr_k\right\| \leq \frac{1}{\lambda^{n + 1}} \max(\varepsilon_\Lambda,2\lambda^2/k),
\]
and if $k \geq 2\lambda^2/\varepsilon_\Lambda$, then
\[
\left\|\sum_{j = 0}^n\xx_j - \vv - \rr_k\right\| \leq \frac{1}{\lambda^{n + 1}} \varepsilon_\Lambda = \frac{\varepsilon_\Lambda}{16\lambda}.
\]
i.e.
\[
\rr_k\in B\left(\sum_{j = 0}^n \xx_j - \vv,\frac{\varepsilon_\Lambda}{16\lambda^n}\right)\cap \left(\bigcup_{q = 1}^{\lambda^n}\frac{\Lambda}{q}\right)\cdot
\]
Since the sequence $(n_k)_1^\infty$ is clearly unbounded, this demonstrates that \eqref{rn4LambdaNn} holds for infinitely many $n$.
\end{subproof}
\end{proof}

Next, we deduce Theorem \ref{theoremdirichletcobounded} as a consequence of Theorem \ref{theoremkhinchincobounded}.

\begin{theorem}[Dirichlet-type theorem]
\label{theoremdirichletcobounded}
Fix $\varepsilon > 0$. For every $\xx\in X$ and for every $q\in\N$, there exists $\pp\in\Lambda$ such that
\begin{equation}
\label{dirichletcobounded}
\left\|\xx - \frac{\pp}{q}\right\| \leq \frac{\codiam(\Lambda) + \varepsilon}{q}\cdot
\end{equation}
In particular, the function $\psi_1(q) = 1/q$ is uniformly Dirichlet, and in fact, $\psi_1$ is optimal.
\end{theorem}
\begin{proof}
Note that
\[
\dist(\xx,\Lambda/q) = \frac{1}{q}\dist(q\xx,\Lambda) \leq \frac{\codiam(\Lambda)}{q}\cdot
\]
Thus for every $\varepsilon > 0$, there exists $\pp/q = \pp_q/q \in\Lambda/q$ such that \eqref{dirichletcobounded} holds. Clearly $\pp_q/q\tendsto q \xx$, which demonstrates that $\psi_1$ is uniformly Dirichlet (with $C = \codiam(\Lambda) + \varepsilon$). Finally, Theorem \ref{theoremkhinchincobounded} implies that $\psi_1$ is optimal.
\end{proof}

We conclude by deducing Theorem \ref{theoremjarnikcobounded} as a corollary of Theorem \ref{theoremjarniknoncobounded}.

\begin{theorem}[Jarn\'ik--Besicovitch type theorem]
\label{theoremjarnikcobounded}
For any nonincreasing function $\psi\to 0$, $\HD(\WA_\psi) = +\infty$. In fact, for any nondecreasing function $f:(0,+\infty)\to(0,+\infty)$, $\HH^f(\WA_\psi) = +\infty$.
\end{theorem}
\begin{proof}
Since the proof of Theorem \ref{theoremjarniknoncobounded} did not assume that $\Lambda$ was not cobounded, to complete the proof it suffices to show that every infinite-dimensional cobounded lattice is not strongly discrete. Suppose that $\Lambda$ is a cobounded lattice, and let $C = 5\codiam(\Lambda)$. Letting $(\ee_i)_1^\infty$ be a sequence of unit vectors with the property \eqref{separated}, we see that the collection $\big(B(3C\ee_i/4,C/4)\big)_{i = 1}^\infty$ is a disjoint collection of subsets of $B(\0,C)$, and thus if $\Lambda$ is strongly discrete then there exists $i\in\N$ such that $\Lambda\cap B(3C\ee_i/4,C/4) = \emptyset$, which implies that $\codiam(\Lambda)\geq C/4$, a contradiction.
\end{proof}

\ignore{

\draftnewpage\section{The Jarn\'ik--Besicovitch Theorem}

For all $\xx\in\WA_\psi$, there exist infinitely many $q\in\N$ such that
\[
\xx \in \bigcup_{\rr\in\Lambda/q} B(\rr,\varepsilon\psi(q)).
\]
Thus by the Hausdorff-Cantelli lemma
\[
\sum_{q = 1}^\infty f(\varepsilon\psi(q))\#(B(\0,Cq)\cap\Lambda) < +\infty \all C > 0 \;\; \Rightarrow \;\; \HH^f(\WA_\psi) = 0.
\]
Note that by regular Jarn\'ik--Besicovitch, we have $\HH^s(W_c) = +\infty$ for all $s,c > 0$. This is consistent with the above since $\#(B(0,C)\cap\Lambda)\gtrsim_{\times,N} C^N$ for all $N\in\N$. (However, this is essentially the only restriction on $C\mapsto \#(B(0,C)\cap\Lambda)$.)

\begin{theorem}
Suppose that $\Lambda$ is not strongly discrete, i.e. that there exists $C > 0$ for which $\#\{B(\0,C)\cap\Lambda\} = +\infty$. If $\psi:\Rplus\to(0,+\infty)$ is nonincreasing and $f:(0,+\infty)\to(0,+\infty)$ is nondecreasing, then
\[
\HH^f(\WA_\psi) = +\infty.
\]
\end{theorem}

}

\ignore{
\draftnewpage\section{Examples}
\label{sectionexamples}

We end this paper with a list of examples and non-examples of lattices in Banach spaces.

\begin{nonexample}
For $1\leq p < +\infty$, $L^p([0,1],\Z)$ is not a lattice in $L^p([0,1],\R)$. (It is not topologically discrete.)
\end{nonexample}

\begin{example}
$L^\infty([0,1],\Z)$ is a cobounded lattice in $L^\infty([0,1],\R)$.
\end{example}

\begin{example}
For $1\leq p < \infty$, $\lb n\ee_n\rb_{n\in\N}$ is a strongly discrete lattice in $\ell^p(\R)$.
\end{example}

\begin{example}
Let $X$ be a Banach space. Let $(\ee_i)_1^\infty$ be a sequence in $X$ whose linear span is dense in $X$. Suppose that there exists a bounded sequence $(\ee_j^*)_1^\infty$ in $X^*$ with the following property:
\[
\ee_j^*[\ee_i] = \delta_{i,j}.
\]
Then $\Lambda := \lb \ee_i\rb_{i = 1}^\infty$ is a lattice in $X$.
\end{example}
\begin{proof}
Condition (II) follows by assumption; to demonstrate (I), suppose that $\pp = \sum_{i\in\N}n_i\ee_i\in\Lambda$, with $n_i\in\Z$, and $n_i = 0$ for all but finitely many $i$. If $\pp\neq 0$, then there exists $i\in\N$ with $n_i\neq 0$; then
\[
1 \leq |n_i| = |\ee_i^*[\pp]| \leq \|\ee_i^*\|\cdot \|\pp\| \leq C\|\pp\|,
\]
where $C = \sup_i \|\ee_i^*\| < +\infty$. Thus $\varepsilon_\Lambda \geq 1/C > 0$, and so $\Lambda$ is discrete.
\end{proof}

\begin{itemize}
\item There exists a cobounded lattice in $\ell^2(\R)$ (see

http://mathoverflow.net/questions/43110/cobounded-cocompact). I believe the argument presented there should work for any uniformly convex Banach space.
\end{itemize}
}

\draftnewpage

\bibliographystyle{amsplain}

\bibliography{bibliography}

\providecommand{\bysame}{\leavevmode\hbox to3em{\hrulefill}\thinspace}
\providecommand{\MR}{\relax\ifhmode\unskip\space\fi MR }
\providecommand{\MRhref}[2]{%
  \href{http://www.ams.org/mathscinet-getitem?mr=#1}{#2}
}
\providecommand{\href}[2]{#2}
\begin{thebibliography}{10}

\bibitem{ADG}
F.~D. Ancel, T.~Dobrowolski, and J.~Grabowski, \emph{Closed subgroups in
  {B}anach spaces}, Studia Math. \textbf{109} (1994), no. 3, 277--290.

\bibitem{CCM}
J.~Chaika, Y.~Cheung, and H.~A. Masur, \emph{Winning games for bounded
  geodesics in moduli spaces of quadratic differentials},
  \url{http://arxiv.org/abs/1109.5976}, preprint 2011.

\bibitem{Christensen}
J.~P.~R. Christensen, \emph{On sets of {H}aar measure zero in abelian {P}olish
  groups}, Israel J. Math. \textbf{13} (1972), 255--260.

\bibitem{DodsonEveritt}
M.~M. Dodson and B.~Everitt, \emph{Metrical {D}iophantine approximation for
  quaternions}, \url{http://arxiv.org/abs/1109.3229}, preprint 2012.

\bibitem{FKMS}
L.~Fishman, D.~Y. Kleinbock, K.~Merrill, and D.~S. Simmons, \emph{Intrinsic
  {D}iophantine approximation on manifolds},
  \url{http://arxiv.org/abs/1405.7650v2}, preprint 2014.

\bibitem{FishmanSimmons1}
L.~Fishman and D.~S. Simmons, \emph{Intrinsic approximation for fractals
  defined by rational iterated function systems - {M}ahler's research
  suggestion}, Proc. Lond. Math. Soc. (3) \textbf{109} (2014), no. 1, 189--212.

\bibitem{FishmanSimmons3}
\bysame, \emph{Unconventional height functions in simultaneous {D}iophantine
  approximation}, \url{http://arxiv.org/abs/1401.8266}, preprint 2014.

\bibitem{FSU4}
L.~Fishman, D.~S. Simmons, and M.~Urba\'nski, \emph{{D}iophantine approximation
  and the geometry of limit sets in {G}romov hyperbolic metric spaces},
  \url{http://arxiv.org/abs/1301.5630}, preprint 2013, to appear Mem. Amer.
  Math. Soc.

\bibitem{Hardy}
G.~H. Hardy, \emph{Orders of infinity. {T}he {I}nfinit\"arcalc\"ul of {P}aul du
  {B}ois-{R}eymond}, Cambridge Tracts in Mathematics and Mathematical Physics,
  No. 12, Hafner Publishing Co., New York, 1971.

\bibitem{HSY}
B.~R. Hunt, T.~D. Sauer, and J.~A. Yorke, \emph{Prevalence: a
  translation-invariant ``almost every'' on infinite-dimensional spaces}, Bull.
  Amer. Math. Soc. \textbf{27} (1992), no. 2, 217--238.

\bibitem{Kristensen}
S.~Kristensen, \emph{On well-approximable matrices over a field of formal
  series}, Math. Proc. Cambridge Philos. Soc. \textbf{135} (2003), no. 2,
  255--268.

\end{thebibliography}

\end{document}